\documentclass[12pt]{article}
\usepackage{mathtools,amsthm,amsfonts,enumitem}
\usepackage[margin=2.5cm]{geometry}
\usepackage[hidelinks]{hyperref}
\usepackage{microtype}
\makeatletter
\renewenvironment{abstract}{%
  \small
  \begin{center}%
    {\bfseries \abstractname\vspace{-.5em}\vspace{\z@}}%
  \end{center}%
  \quotation
}{%
  \endquotation
}
\makeatother

\numberwithin{equation}{section}
\theoremstyle{plain}
\newtheorem{theorem}{Theorem}[section]
\newtheorem{proposition}[theorem]{Proposition}
\newtheorem{lemma}[theorem]{Lemma}
\newtheorem{corollary}[theorem]{Corollary}
\theoremstyle{definition}
\newtheorem{definition}[theorem]{Definition}
\theoremstyle{remark}
\newtheorem{remark}[theorem]{Remark}

\providecommand{\cP}{\mathcal{P}}
\providecommand{\R}{\mathbb{R}}
\providecommand{\cF}{\mathcal{F}}
\providecommand{\cZ}{\mathcal{Z}}
\providecommand{\cQ}{\mathcal{Q}}
\providecommand{\ind}[1]{\mathbf{1}_{\{#1\}}}
\providecommand{\TV}{\mathrm{TV}}
\providecommand{\KL}{\mathrm{KL}}
\DeclareMathOperator{\supp}{supp}
\DeclareMathOperator{\dist}{dist}
\providecommand{\Xd}{X_{\mathrm d}}
\providecommand{\Fd}{F_{\mathrm d}}
\providecommand{\mud}{\mu^{\mathrm d}}
\providecommand{\Cd}{C^{\mathrm d}}
\providecommand{\Ped}{P^{\mathrm d}_\varepsilon}
\providecommand{\Pdstar}{(P^{\mathrm d})^\star}
\providecommand{\Gammad}{\Gamma^{\mathrm d}}
\providecommand{\Id}{I_{\Gamma^{\mathrm d}}}

\title{Failure of ambient closed-set large-deviation upper bounds in entropic optimal transport}
\author{Maja Gw\'o\'zd\'z\\ ETH Z\"urich\\ \texttt{mgwozdz@ethz.ch}}
\date{}
\makeatletter
\renewcommand{\maketitle}{%
  \begin{center}%
    {\Large\normalfont\@title\par}%
    \vspace{1em}%
    {\normalsize\normalfont\@author\par}%
  \end{center}%
  \vspace{1.5em}%
}
\makeatother
\begin{document}
\maketitle
\pagestyle{plain}
\begin{abstract}
Large-deviation upper bounds on compact sets do not, in general, extend to arbitrary closed sets without additional tightness. We show that this obstruction already occurs in static entropic optimal transport. More precisely, we construct a fixed-cost model with continuous cost and nonatomic marginals for which the entropic minimisers converge in total variation to an optimal plan with noncompact support, the known compact-set upper bound remains valid, but the corresponding closed-set upper bound fails on a specific closed subset of the ambient space. For a fixed closed set, we identify the exact tail criterion for passing from compact to closed sets. We show that there does not exist a full large-deviation principle (LDP) on the ambient space at speed \(1/\varepsilon\) with an arbitrary lower semicontinuous rate function.
\end{abstract}

\section{Introduction and main results}\label{sec:intro}

We study the zero-temperature limit of entropic optimal transport (EOT), that is, the static Schr\"odinger problem. Let \(X\) and \(Y\) be Polish spaces, let \(\mu\in\cP(X)\), \(\nu\in\cP(Y)\), and let \(C:X\times Y\to[0,\infty)\) be continuous. We write \(\Pi(\mu,\nu)\) for the set of Borel probability measures on \(X\times Y\), where the first marginal is \(\mu\) and the second one \(\nu\). For probability measures \(P,Q\) on some common measurable space, we write
\[
\KL(P\|Q):=
\begin{cases}
\displaystyle \int \log\!\left(\frac{dP}{dQ}\right)\,dP, & P\ll Q,\\[1ex]
+\infty, & \text{otherwise},
\end{cases}
\]
for the relative entropy of \(P\) with respect to \(Q\). We assume that there exists \(P_0\in\Pi(\mu,\nu)\) such that
\[
\int_{X\times Y} C\,dP_0<\infty
\qquad\text{and}\qquad
\KL(P_0\|\mu\otimes\nu)<\infty.
\]
Notice that under this assumption, for each \(\varepsilon>0\), the functional
\[
P\longmapsto \int_{X\times Y} C\,dP+\varepsilon\,\KL(P\|\mu\otimes\nu)
\]
admits a unique minimiser \(P_\varepsilon\in\Pi(\mu,\nu)\). We shall study the limit \(\varepsilon\downarrow0\). Let \(\varepsilon_n\downarrow0\), suppose \(P_{\varepsilon_n}\Rightarrow P^\star\), and set \(\Gamma:=\supp(P^\star)\). Bernton, Ghosal, and Nutz (BGN) showed \cite[Cor.~4.3]{BGN} that
\[
\limsup_{n\to\infty}\varepsilon_n\log P_{\varepsilon_n}(K)\le -\inf_K I_\Gamma,
\]
for every compact \(K\subset X\times Y\). One natural question is whether the same upper bound remains valid for arbitrary closed sets when \(\Gamma\) is noncompact.
We use \(I_\Gamma\) to denote the associated BGN rate \cite[Eq.~(4.3)]{BGN} (defined in \eqref{eq:BGNrate-prelim} below). Globally, exponential tightness ``upgrades'' compact-set upper bounds to closed-set upper bounds (\textit{cf}. \cite[Lem.~1.2.18(a)]{DemboZeitouni}). Recall that a family \((\cQ_\varepsilon)_{\varepsilon>0}\subset\cP(\cZ)\) is exponentially tight at speed \(1/\varepsilon\) if for every \(M<\infty\), there exists a compact set \(K_M\subset\cZ\) such that
\[
\limsup_{\varepsilon\downarrow0}\varepsilon\log \cQ_\varepsilon(K_M^c)\le -M.
\]
For a fixed closed set \(F\), we determine the exact tail exponent for this upgrade in Proposition~\ref{prop:tail-criterion}. In informal terms, it is a one-set version of the standard compact-plus-tail argument. In the fixed-marginal EOT setting, exponential tightness already fails whenever one marginal has noncompact support (see Proposition~\ref{prop:no-exp-tight-fixed-marg}). In Theorem~\ref{thm:main-cont}, we show that this obstruction is realised inside the fixed-cost entropic optimal transport with continuous cost. Indeed, the compact-set upper bound of \cite[Cor.~4.3]{BGN} remains valid on every compact set, but it need not extend to arbitrary closed sets.

\begin{theorem}[A nonatomic counterexample to the closed-set upper bound]\label{thm:main-cont}
There exist a closed subset \(X\subset\R\), a nonatomic measure \(\mu\in\cP(X)\), and a continuous cost \(C:X\times X\to[0,\infty)\) for which the entropic minimisers \(P_\varepsilon\in\Pi(\mu,\mu)\) satisfy the following properties:
\begin{enumerate}[label=\textup{(\roman*)},ref=\textup{(\roman*)},leftmargin=2.2em]
\item\label{item:main-cont-tv}
\(P_\varepsilon\to P^\star\) in total variation for an optimal plan \(P^\star\) with noncompact support \(\Gamma:=\supp(P^\star)\),
\item\label{item:main-cont-closed-set}
for some closed noncompact set \(F\subset X\times X\),
\[
\inf_F I_\Gamma>0
\qquad\text{but}\qquad
\limsup_{\varepsilon\downarrow0}\varepsilon\log P_\varepsilon(F)>-\inf_F I_\Gamma.
\]
\end{enumerate}
\end{theorem}
\begin{theorem}[No full LDP on \(X\times X\)]\label{thm:intro-no-ambient}
In the context of Theorem~\ref{thm:main-cont}, there does not exist a lower semicontinuous rate function \(I_{\mathrm{amb}}:X\times X\to[0,\infty]\) such that \((P_\varepsilon)_{\varepsilon>0}\) satisfies a full LDP on \(X\times X\) at speed \(1/\varepsilon\).
\end{theorem}
The selected limit is a block-product plan that we obtain by replacing each discrete label by a compact interval (see \eqref{eq:Pstar-cont} for the formula). Its support is a countable union of diagonal rectangles that escape to infinity, which implies that the support is noncompact. The obstruction is stronger than the failure of the support-based candidate \(I_\Gamma\) on one closed set. Indeed, in the constructed example, there does not exist a lower semicontinuous ambient rate function that could give a full LDP on \(X\times X\), even though the BGN compact-set upper bound \cite[Cor.~4.3]{BGN} remains valid on every compact set.

The construction is deliberately topological rather than rigid in geometric terms. More precisely, the space \(X\) is a disconnected closed subset of \(\R\), and the cost \(C\) is a blockwise constant. In this work, we identify an ambient-space obstruction to upgrading compact-set upper bounds, but it does not rule out the possibility that additional geometric hypotheses, such as connected supports, stronger coercivity, or more rigid ambient geometry, could guarantee a full LDP.

\subsection{Idea of the construction}\label{subsec:idea}
The construction starts from a discrete block model. In block \(n\), the \(m_n(m_n-1)\) ordered off-diagonal states all have cost \(b\). In heuristic terms, the symmetry puts each such state on the scale \(w_n m_n^{-1}e^{-b/\varepsilon_n}\). When \(m_n\asymp e^{\kappa n}\), the block mass is on the scale \(w_n e^{-(b-\kappa)/\varepsilon_n}\). In Section~\ref{sec:discrete} and Appendix~\ref{app:block-bookkeeping}, we make this heuristic more precise. It follows that the multiplicity lowers the effective exponent from \(b\) to \(b-\kappa\). This mechanism is related to the exponential-penalty asymptotics in linear programming. For more details, compare the exponential-penalty analysis of \cite{CominettiSanMartin} and, for the direct entropic-penalty context with explicit nonasymptotic bounds, see \cite{Weed2018}. We obtain the nonatomic example by replacing each discrete label with a compact interval that is clopen in the relative topology of \(X\) and by taking the cost to be constant on label rectangles. By the entropy chain rule, we then show that the continuous Schr\"odinger problem is the exact lift of the discrete one.

The paper is organised as follows. In Section~\ref{sec:prelim}, we develop the general tail criterion and the obstruction related to the escaping set. In Section~\ref{sec:discrete}, we construct the discrete counterexample. Section~\ref{sec:nonatomic} then lifts the construction to a nonatomic model and proves the that a full LDP on \(X\times X\) cannot exist. In Appendix~\ref{app:superlevel-tail}, we give an auxiliary sufficient criterion that guarantees the closed-set upper bound. Finally, in Appendices~\ref{app:schroedinger}--\ref{app:routine-checks}, we present some relevant discrete Schr\"odinger facts and the basic parameter checks.

\subsection{Related works}\label{subsec:related}
For background on the static Schr\"odinger problem and its relation to EOT, see \cite{LeonardSurvey}. For convergence towards optimal transport, once the regularisation vanishes, see \cite{LeonardJFA}. For the squared-Euclidean context, consult also \cite{CarlierDuvalPeyreSchmitzer}. For the fixed-cost static problem, BGN prove the compact-set upper bound and note that, when \(P^\star=\lim_{\varepsilon\downarrow0}P_\varepsilon\) is not compactly supported, it is open whether this corollary extends from compact to arbitrary closed sets \cite[Cor.~4.3]{BGN}. In Theorem~\ref{thm:main-cont}, we give a negative answer in that original context. In our example, the BGN compact-set upper bound \cite[Cor.~4.3]{BGN} is still valid on every compact set, but it cannot be upgraded to all closed sets once the selected support is noncompact.

For compactly supported \(L^\infty\) marginals and broad classes of costs, Carlier, Pegon, and Tamanini identify the behaviour of the difference between the entropic transport value and the unregularised transport value \cite{CarlierPegonTamaniniValue}. Eckstein and Nutz prove sharp convergence rates for general divergence-regularised transport problems, which include the entropic case \cite{EcksteinNutzQuantization}. Malamut and Sylvestre derive convergence rates for the regularised plans, transport cost, and entropy term in the static entropic problem \cite{MalamutSylvestreRates}, and L\'opez-Rivera proves a uniform convergence rate for entropic potentials in the quadratic Euclidean context \cite{LopezRiveraUniformPotentials}.

A different open problem is described in \cite{NilssonNyquistReflection}, where the reference measures vary with the noise level and are given by costs \(c_\eta\to c\). In that setting with variable cost, the main result is a ``weak-type'' large-deviation statement rather than a weak LDP in the standard sense, while compact support of both marginals restores the full LDP \cite[Thm.~4.2, Cor.~4.3]{NilssonNyquistReflection}. Our obstruction is stronger than the failure of exponential tightness noted in \cite[Rem.~4.2(b)]{NutzWieselPotentials}. It also does not contradict \cite[Cor.~1]{KatoDynamic}, which gives a full LDP in the Brownian dynamical setting when both endpoint supports are compact and one of them is the closure of a connected open set. On compact spaces, Nutz and Wiesel show that convergence of Schr\"odinger potentials provides upper and lower bounds for arbitrary measurable sets in terms of the essential infimum of \(c-f-g\), where \(f,g\) denote the limiting potentials and \(c-f-g\) denotes the function \((x,y)\mapsto c(x,y)-f(x)-g(y)\), and that a full LDP determines the limiting potentials \cite[Prop.~4.1, Rem.~4.2(a), and Prop.~4.5]{NutzWieselPotentials}. For the dynamic or variable-cost large deviations beyond the static context with fixed costs, see also \cite{KatoDynamic,NilssonNyquistReflection,NilssonNyquistDynamic}. For an orthogonal issue on the entropic selection problem for non-strictly convex costs, see \cite{DiMarinoLouet,LeyLineSelection} and \cite{AryanGhosalSelection} for the Euclidean distance on disjoint compact supports in higher dimensions.

\section{Compact-set upper bounds and tail criteria}\label{sec:prelim}

\subsection{The BGN rate and the compact upper bound}\label{subsec:bgn-application}
Under the assumptions of Section~\ref{sec:intro}, recall that
\[
J_\varepsilon(P):=\int C\,dP + \varepsilon\,\KL(P\|\mu\otimes\nu),\qquad P\in\Pi(\mu,\nu).
\]
For probability measures \(P,Q\) on a common measurable space, we write
\[
\|P-Q\|_{\TV}:=\sup_A |P(A)-Q(A)|=\frac12\int |dP-dQ|
\]
for the total variation distance. By the finiteness assumption, if
\[
dR_\varepsilon:=Z_\varepsilon^{-1}e^{-C/\varepsilon}\,d(\mu\otimes\nu),
\]
then
\[
J_\varepsilon(P)=\varepsilon\,\KL(P\|R_\varepsilon)-\varepsilon\log Z_\varepsilon.
\]
This implies that \(J_\varepsilon\) has a unique minimiser \(P_\varepsilon\in\Pi(\mu,\nu)\) for every \(\varepsilon>0\) by \cite[Prop.~A.1]{BGN}. We follow \cite{BGN}, and define $I_\Gamma:X\times Y\to[0,\infty]$ by
\begin{equation}\label{eq:BGNrate-prelim}
\begin{aligned}
I_\Gamma(x,y)
:={}&\sup_{k\ge 2}\ 
\sup_{\substack{(x_i,y_i)_{i=2}^k\in \Gamma^{k-1}}}\ 
\sup_{\sigma\in \Sigma(k)}
\Bigg[
\sum_{i=1}^k C(x_i,y_i)
-\sum_{i=1}^k C(x_i,y_{\sigma(i)})
\Bigg],\\
&\hspace{7.5em} (x_1,y_1)=(x,y).
\end{aligned}
\end{equation}
where $\Sigma(k)$ denotes the permutations of $\{1,\dots,k\}$. For an arbitrary sequence \(\varepsilon_k\downarrow0\) with \(P_{\varepsilon_k}\Rightarrow P^\star\), we set \(\Gamma:=\supp(P^\star)\) and take \(I_\Gamma\) from \eqref{eq:BGNrate-prelim} as the associated BGN candidate rate. By \cite[Prop.~3.2]{BGN}, every such limit \(P^\star\) is an optimal transport plan for \(C\). In particular, \(\Gamma\) is \(C\)-cyclically monotone.

\begin{lemma}[Two-point lower bound for \(I_\Gamma\)]\label{lem:I-two-point}
For every \((x,y)\in X\times Y\),
\[
I_\Gamma(x,y)\ \ge\ \sup_{(x',y')\in\Gamma}\Big( C(x,y)+C(x',y')-C(x,y')-C(x',y)\Big).
\]
\end{lemma}
\begin{proof}
In \eqref{eq:BGNrate-prelim}, it suffices to take \(k=2\), choose \((x_2,y_2)=(x',y')\in\Gamma\), and let \(\sigma\) be the transposition that exchanges \(1\) and \(2\).
The bracket in \eqref{eq:BGNrate-prelim} is then equal to \(C(x,y)+C(x',y')-C(x,y')-C(x',y)\). Finally, we take the supremum over \((x',y')\in\Gamma\) and obtain the claim.
\end{proof}

\begin{corollary}[Compact-set upper bound under full convergence]\label{cor:BGN-full-limit}
Let us assume that \(P_\varepsilon\Rightarrow P^\star\) as \(\varepsilon\downarrow0\), and set \(\Gamma:=\supp(P^\star)\). It follows that for every compact \(K\subset X\times Y\),
\[
\limsup_{\varepsilon\downarrow0}\varepsilon\log P_\varepsilon(K)\le -\inf_K I_\Gamma.
\]
\end{corollary}
\begin{proof}
Let us choose a sequence \(\varepsilon_n\downarrow0\) with
\[
\varepsilon_n\log P_{\varepsilon_n}(K)\longrightarrow
\limsup_{\varepsilon\downarrow0}\varepsilon\log P_\varepsilon(K).
\]
Given that \(P_{\varepsilon_n}\Rightarrow P^\star\) holds, we may apply Corollary~4.3 of \cite{BGN} directly along the subsequence \((P_{\varepsilon_n})\). We then obtain:
\[
\limsup_{n\to\infty}\varepsilon_n\log P_{\varepsilon_n}(K)\le -\inf_K I_\Gamma.
\]
\end{proof}

\subsection{A fixed closed-set tail criterion}\label{subsec:tail-prelim}

In this subsection, we determine the tail exponent that allows for passing from compact to closed sets.

\begin{proposition}[Closed-set criterion]\label{prop:tail-criterion}
Let \(\cZ\) be Polish and \((\cQ_\varepsilon)_{\varepsilon>0}\subset\cP(\cZ)\).
Let us further assume there exists \(I:\cZ\to[0,\infty]\) such that for every compact \(K\subset\cZ\),
\[
\limsup_{\varepsilon\downarrow0}\varepsilon\log \cQ_\varepsilon(K)\le -\inf_K I.
\]
For a closed set \(\cF\subset\cZ\), we define
\[
\beta(\cF):=
\inf_{K\subset\cZ\ \mathrm{compact}}
\limsup_{\varepsilon\downarrow0}\varepsilon\log \cQ_\varepsilon(\cF\cap K^c)
\in[-\infty,0].
\]
This implies that
\[
\limsup_{\varepsilon\downarrow0}\varepsilon\log \cQ_\varepsilon(\cF)
\le
\max\bigl\{-\inf_{\cF} I,\ \beta(\cF)\bigr\}.
\]
In particular, the closed-set upper bound
\[
\limsup_{\varepsilon\downarrow0}\varepsilon\log \cQ_\varepsilon(\cF)\le -\inf_{\cF}I
\]
holds if and only if
\[
\beta(\cF)\le -\inf_{\cF} I.
\]
\end{proposition}
\begin{proof}
We begin by fixing a closed set \(\cF\subset\cZ\) and a compact set \(K\subset\cZ\).
Observe that \(\cF\cap K\) is compact, so by our assumption on the compact-set upper bound, we obtain
\[
\limsup_{\varepsilon\downarrow0}\varepsilon\log \cQ_\varepsilon(\cF\cap K)
\le -\inf_{\cF\cap K} I
\le -\inf_{\cF} I.
\]
We now apply
\[
\cQ_\varepsilon(\cF)\le \cQ_\varepsilon(\cF\cap K)+\cQ_\varepsilon(\cF\cap K^c)
\]
and
\[
\varepsilon\log(u+v)\le \varepsilon\log 2+\max\{\varepsilon\log u,\varepsilon\log v\},
\]
which gives
\[
\limsup_{\varepsilon\downarrow0}\varepsilon\log\cQ_\varepsilon(\cF)
\le
\max\Bigl\{-\inf_{\cF} I,\ \limsup_{\varepsilon\downarrow0}\varepsilon\log \cQ_\varepsilon(\cF\cap K^c)\Bigr\}.
\]
We now take the infimum over compact \(K\) to obtain:
\[
\limsup_{\varepsilon\downarrow0}\varepsilon\log \cQ_\varepsilon(\cF)
\le
\max\bigl\{-\inf_{\cF} I,\ \beta(\cF)\bigr\}.
\]
If \(\beta(\cF)\le -\inf_{\cF} I\), the bound gives the closed-set upper bound on \(\cF\).
On the other hand, if the closed-set upper bound holds on \(\cF\), then
\[
\cQ_\varepsilon(\cF\cap K^c)\le \cQ_\varepsilon(\cF)
\]
for every compact \(K\), so
\[
\beta(\cF)\le -\inf_{\cF} I
\]
follows after we take \(\limsup_{\varepsilon\downarrow0}\varepsilon\log\) and then the infimum over \(K\).
\end{proof}
Note that Proposition~\ref{prop:tail-criterion} is a specialisation, applied to fixed sets, of the standard compact-plus-tail argument in large-deviation theory.

\begin{proposition}[Failure on closed sets]\label{prop:escaping-obstruction}
Let \(\cZ\) be Polish and \((\cQ_\varepsilon)_{\varepsilon>0}\subset\cP(\cZ)\).
We assume that \(I:\cZ\to[0,\infty]\) satisfies the compact-set upper bound
\[
\limsup_{\varepsilon\downarrow0}\varepsilon\log \cQ_\varepsilon(K)\le -\inf_K I
\qquad\text{for every compact }K\subset\cZ.
\]
Let \(F\subset\cZ\) be closed, and suppose that there exist a sequence \(\varepsilon_n\downarrow0\), Borel sets \(E_n\subset F\), and a number \(s\ge0\) such that
\begin{enumerate}[label=\textup{(\roman*)},ref=\textup{(\roman*)},leftmargin=2.2em]
\item\label{item:escaping-obstruction-compacts}
for every compact \(K\subset\cZ\), we have \(E_n\subset K^c\) for all sufficiently large \(n\),
\item
\[
\limsup_{n\to\infty}\varepsilon_n\log \cQ_{\varepsilon_n}(E_n)\ge -s.
\]
\end{enumerate}
This implies that
\[
\beta(F)\ge -s.
\]
In particular, if \(s<\inf_F I\), then
\[
\beta(F)>-\inf_F I,
\]
which shows that the closed-set upper bound fails on \(F\).
\end{proposition}
\begin{proof}
Let us fix a compact set \(K\subset\cZ\). By \ref{item:escaping-obstruction-compacts}, \(E_n\subset F\cap K^c\) holds for all sufficiently large \(n\). It follows that
\begin{align*}
\limsup_{\varepsilon\downarrow0}\varepsilon\log \cQ_\varepsilon(F\cap K^c)
&\ge
\limsup_{n\to\infty}\varepsilon_n\log \cQ_{\varepsilon_n}(F\cap K^c)\\
&\ge
\limsup_{n\to\infty}\varepsilon_n\log \cQ_{\varepsilon_n}(E_n)\\
&\ge -s.
\end{align*}
It suffices to take the infimum over compact \(K\) to obtain \(\beta(F)\ge -s\). If \(s<\inf_F I\), then \(-s>-\inf_F I\), so \(\beta(F)>-\inf_F I\). The failure of the closed-set upper bound now follows immediately from Proposition~\ref{prop:tail-criterion}.
\end{proof}

\begin{remark}[Exponential-tightness upgrade]\label{rem:exp-tight-upgrade}
Under the assumptions of Proposition~\ref{prop:tail-criterion}, exponential tightness implies
\[
\limsup_{\varepsilon\downarrow0}\varepsilon\log \cQ_\varepsilon(\cF)\le -\inf_{\cF} I
\]
for every closed \(\cF\subset\cZ\) (see \cite[Lemma~1.2.18(a)]{DemboZeitouni}). In other terms, Proposition~\ref{prop:tail-criterion} gives \(\beta(\cF)=-\infty\) for all closed \(\cF\).
\end{remark}
Note that exponential tightness is impossible in the fixed-marginal noncompact regime that we consider here (see Proposition~\ref{prop:no-exp-tight-fixed-marg} below).

\begin{proposition}[Fixed noncompact marginals and exponential tightness]\label{prop:no-exp-tight-fixed-marg}
Let \((\cQ_\varepsilon)_{\varepsilon>0}\subset \Pi(\mu,\nu)\). If \(\supp(\mu)\) is noncompact or \(\supp(\nu)\) is noncompact, then
\[
\limsup_{\varepsilon\downarrow0}\varepsilon\log \cQ_\varepsilon(K^c)=0
\]
for every compact \(K\subset X\times Y\). In particular, \((\cQ_\varepsilon)\) is not exponentially tight at speed \(1/\varepsilon\).
\end{proposition}
\begin{proof}
Let us assume that \(\supp(\mu)\) is noncompact (the \(\nu\)-case is symmetric).
We fix compact \(K\subset X\times Y\), and set \(K_X:=\mathrm{proj}_1(K)\), which is compact in \(X\). If \(\mu(K_X^c)=0\), then \(\supp(\mu)\subset K_X\), which contradicts noncompactness of \(\supp(\mu)\), so \(\mu(K_X^c)>0\) follows.
Since every \(\cQ_\varepsilon\) has first marginal \(\mu\), we have
\[
\cQ_\varepsilon(K^c)\ge \cQ_\varepsilon(K_X^c\times Y)=\mu(K_X^c)>0.
\]
This implies \(\limsup_{\varepsilon\downarrow0}\varepsilon\log \cQ_\varepsilon(K^c)\ge 0\). The reverse inequality follows immediately from \(\cQ_\varepsilon(K^c)\le1\).
\end{proof}

\begin{remark}
The failure of exponential tightness for fixed marginals with noncompact support is already implicit in the literature. For the static problem, see \cite[Rem.~4.2(b)]{NutzWieselPotentials}, and consult \cite[Rem.~6]{KatoDynamic} for the dynamical Brownian setting.
\end{remark}

\begin{remark}
Note that failure of exponential tightness alone does not exclude a full LDP with a non-good rate function on a noncompact space. Indeed, if \(\cQ_\varepsilon\equiv \cQ\) for all \(\varepsilon>0\), then \((\cQ_\varepsilon)_{\varepsilon>0}\) satisfies a full LDP on \(\cZ\) with rate \(I=0\) on \(\supp(\cQ)\) and \(I=+\infty\) on \(\cZ\setminus\supp(\cQ)\). Proposition~\ref{prop:no-exp-tight-fixed-marg} only rules out the standard exponential-tightness upgrade from compact sets to closed sets. The ambient no-LDP theorem we prove below is strictly stronger.
\end{remark}

\section{A discrete counterexample}\label{sec:discrete}
We begin with a discrete model, which already shows the full tail mechanism behind the counterexample. In Section~\ref{sec:nonatomic}, we lift the construction to a nonatomic setting.

\begin{theorem}[A discrete counterexample to the closed-set upper bound]\label{thm:main-discrete}
There exist a countable discrete Polish space \(\Xd\), a probability measure \(\mu^{\mathrm d}\in\cP(\Xd)\), and a continuous cost \(C^{\mathrm d}:\Xd\times\Xd\to[0,\infty)\) for which the entropic minimisers \(P_\varepsilon^{\mathrm d}\in\Pi(\mu^{\mathrm d},\mu^{\mathrm d})\) satisfy the following properties.
\begin{enumerate}[label=\textup{(\roman*)},ref=\textup{(\roman*)},leftmargin=2.2em]
\item\label{item:main-discrete-tv}
The unregularised optimal transport problem has the unique minimiser \((P^{\mathrm d})^\star\), that is, the diagonal coupling. Moreover, \(P^{\mathrm d}_\varepsilon\to (P^{\mathrm d})^\star\) in total variation as \(\varepsilon\downarrow0\), and \(\Gamma^{\mathrm d}:=\supp((P^{\mathrm d})^\star)\) is noncompact.
\item\label{item:main-discrete-closed-set}
There exists a closed noncompact set $\Fd\subset \Xd\times\Xd$ such that $\inf_{\Fd} I_{\Gamma^{\mathrm d}}>0$ but
\[
\limsup_{\varepsilon\downarrow0}\ \varepsilon\log P^{\mathrm d}_\varepsilon(\Fd)\;>\;-\inf_{\Fd} I_{\Gamma^{\mathrm d}}.
\]
\end{enumerate}
\end{theorem}
\begin{remark}
The stronger obstruction that we show later in Theorem~\ref{thm:intro-no-ambient} can already be noticed in the discrete case. We defer this discussion until after the nonatomic lift because the main point of the paper is that the pathology persists for the continuous cost and nonatomic marginals.
\end{remark}
\subsection{Construction of \texorpdfstring{$(\Xd,\mud,\Cd)$}{(Xd,mud,Cd)}}\label{subsec:discrete-construction}

We fix the parameters:
\begin{equation}\label{eq:params}
0<2a<\kappa<b.
\end{equation}
The relevant inequality is \(b-\kappa<b-2a\). Observe that the closed test set carries mass at exponential rate \(b-\kappa\), but the BGN rate on that set is bounded below by \(b-2a\). We further construct the following sequences
\begin{equation}\label{eq:concrete-seqs}
m_n:=\max\big\{2,\big\lfloor e^{\kappa n}\big\rfloor\big\},\qquad
L_n:=n^4,\qquad
w_n:=\frac{1}{\zeta(3)}\frac{1}{n^3},\qquad n\ge1.
\end{equation}
We verify the elementary asymptotic and integrability properties in Appendix~\ref{app:routine-checks}. In particular,
\[
\begin{gathered}
\sum_{n\ge1} w_n\log\frac{m_n+1}{w_n}<\infty,\qquad
L_n\to\infty,\\
m_n e^{-\kappa n}\to1,\qquad
\frac{\varepsilon_n}{L_n w_n}\to0
\end{gathered}
\]
for \(\varepsilon_n:=1/n\).
For each $n$, we define a finite block
\[
B_n:=\{(n,i): i\in\{0,1,\dots,m_n\}\},
\]
and set
\[
\Xd:=\bigsqcup_{n\ge 1} B_n
=\{(n,i): n\ge 1,\ i\in\{0,1,\dots,m_n\}\},
\]
which is equipped with the discrete metric $d(x,y):=\ind{x\neq y}$. Note that $\Xd$ is countable, $(\Xd,d)$ is separable and complete, hence Polish. We call \((n,0)\) \textit{the distinguished point} of \(B_n\), and \((n,1),\dots,(n,m_n)\) \textit{the high-multiplicity points}. Let us define a probability measure \(\mud\in\cP(\Xd)\) by
\begin{equation}\label{eq:mu}
\mud(\{(n,i)\})=\frac{w_n}{m_n+1},\qquad n\ge 1,\;0\le i\le m_n.
\end{equation}
It follows that \(\mud(B_n)=w_n\) and \(\sum_n w_n=1\).

We also define a symmetric cost \(\Cd:\Xd\times\Xd\to[0,\infty)\) by
\begin{equation}\label{eq:cost}
\begin{aligned}
&\Cd((n,i),(m,j))\\
&\quad :=
\begin{cases}
0, & n=m\text{ and }i=j,\\
a, & n=m\text{ and exactly one of }i,j\text{ equals }0,\\
b, & \substack{n=m,\ i\neq j,\\ i,j\in\{1,\dots,m_n\}},\\
L_n+L_m, & n\neq m.
\end{cases}
\end{aligned}
\end{equation}
Notice that \(\Xd\) is discrete, so \(\Cd\) is continuous. Finally, we define the test set
\begin{equation}\label{eq:Fdef}
\Fd:=\bigcup_{n\ge1}\Big\{((n,i),(n,j)): 1\le i\neq j\le m_n\Big\}\subset \Xd\times\Xd.
\end{equation}
This means that \(\Fd\) consists of the within-block off-diagonal pairs among the high-multiplicity points. Given that the topology is discrete, \(\Fd\) is clopen. It is also infinite, hence noncompact by Lemma~\ref{lem:compact-finite}. In the rest of Section~\ref{sec:discrete}, we write \(\Lambda^{\mathrm d}:=\mud\otimes\mud\).

\subsection{The discrete Schr\"odinger problem}\label{subsec:discrete-eot}

In Section~\ref{sec:discrete}, we use \(J_\varepsilon^{\mathrm d}\) to denote the Schr\"odinger functional for \((\Xd,\mud,\mud,\Cd)\) with reference measure \(\Lambda^{\mathrm d}:=\mud\otimes\mud\). Whenever it is convenient, we use
\[
\KL(P\|\Lambda^{\mathrm d})
=\sum_{x,y\in\Xd} P(x,y)\log\frac{P(x,y)}{\Lambda^{\mathrm d}(x,y)},
\qquad 0\log 0:=0.
\]
For each \(\varepsilon>0\), let \(\Ped\) be the minimiser of \(J_\varepsilon^{\mathrm d}\), and let \(\Pdstar(x,y):=\mud(x)\ind{x=y}\) denote the diagonal coupling. In Appendix~\ref{app:schroedinger}, we prove the existence and uniqueness of \(\Ped\), its factorisation and strict positivity, and the symmetry and cross-ratio identities that we use below.

\begin{lemma}[Assumptions]\label{lem:Hmu}
For the model of Subsection~\ref{subsec:discrete-construction}, \(\Xd\) is Polish, \(\mud\in\cP(\Xd)\), and \(\Cd\ge0\) is continuous.
We define
\[
H(\mud):=\sum_{x\in\Xd}\mud(x)\log\frac{1}{\mud(x)}.
\]
It follows that \(H(\mud)<\infty\), and for the diagonal coupling \(\Pdstar\) with \(\Lambda^{\mathrm d}:=\mud\otimes\mud\), we have
\[
\KL(\Pdstar\|\Lambda^{\mathrm d})=H(\mud),
\]
so \(J_\varepsilon^{\mathrm d}(\Pdstar)=\varepsilon H(\mud)<\infty\) holds for every \(\varepsilon>0\).
In particular, the finiteness assumptions from the general setup in Section~\ref{sec:intro} also hold for the discrete model.
\end{lemma}
\begin{proof}
We established continuity and Polishness in Subsection~\ref{subsec:discrete-construction}. The identity follows immediately from \(\Lambda^{\mathrm d}(x,y)=\mud(x)\mud(y)\) and \(\Pdstar(x,y)=\mud(x)\ind{x=y}\).
Finally, we obtain
\[
H(\mud)=\sum_{n\ge1}\sum_{i=0}^{m_n}\frac{w_n}{m_n+1}\log\frac{m_n+1}{w_n}
=\sum_{n\ge1}w_n\log\frac{m_n+1}{w_n}<\infty
\]
by Lemma~\ref{lem:Hmu-proof-app}.
\end{proof}

\subsection{Block estimates}\label{subsec:block-estimate-statement}
For the special sequence \(\varepsilon_n:=1/n\), let us write
\[
\begin{aligned}
P_n^{\mathrm d}&:=P_{\varepsilon_n}^{\mathrm d},\\
F_n&:=\{((n,i),(n,j)):1\le i\neq j\le m_n\}\subset\Fd,\\
\Delta_n&:=P_n^{\mathrm d}(B_n\times B_n^c).
\end{aligned}
\]
For a fixed representative \(i=1\), we set
\[
\begin{aligned}
d_n&:=P_n^{\mathrm d}((n,1),(n,1)),\\
o_n&:=P_n^{\mathrm d}((n,1),(n,2)),\\
r_n&:=P_n^{\mathrm d}((n,1),(n,0)),\\
e_n&:=\sum_{y\notin B_n}P_n^{\mathrm d}((n,1),y).
\end{aligned}
\]
By Lemma~\ref{lem:o-vs-d-app}, we obtain \(o_n=d_n e^{-bn}\).

\begin{proposition}[Block estimates]\label{prop:block-estimate}
We use the same notation as above. Let us consider the following estimates.
\begin{enumerate}[label=\textup{(\alph*)},ref=\textup{(\alph*)},leftmargin=2.2em]
\item\label{item:block-estimate-delta}
For all \(n\),
\[
\Delta_n\le \frac{\varepsilon_n H(\mud)}{2L_n}.
\]
\item\label{item:block-estimate-diagonal}
For all sufficiently large \(n\),
\[
d_n\ge \frac14\,\mud(\{(n,1)\})=\frac{w_n}{4(m_n+1)}.
\]
\item\label{item:block-estimate-mass}
For all sufficiently large \(n\),
\[
P_n^{\mathrm d}(\Fd)\ge P_n^{\mathrm d}(F_n)=m_n(m_n-1)o_n\ge \frac{w_n}{16}\,e^{-(b-\kappa)n}.
\]
\end{enumerate}
We defer the proofs of these estimates to Appendix~\ref{app:block-bookkeeping}.
\end{proposition}

\subsection{The unregularised minimiser and its support}\label{subsec:discrete-ot}

\begin{proposition}[Unique unregularised minimiser]\label{prop:OTunique}
The unique optimal transport plan for \((\mud,\mud)\) and cost \eqref{eq:cost} is the diagonal coupling \(\Pdstar\).
Its support
\[
\Gammad:=\supp(\Pdstar)=\{(x,x):x\in \Xd\}
\]
is noncompact.
\end{proposition}
\begin{proof}
We set
\[
C_0:=\inf\{\Cd(x,y):x,y\in\Xd,\ x\neq y\}>0.
\]
By construction, \(\Cd(x,y)=0\) holds if and only if \(x=y\).
For an arbitrary coupling \(P\in\Pi(\mud,\mud)\), we get
\[
\sum_{x,y} \Cd(x,y)P(x,y)\ge C_0\,P(\{x\neq y\})\ge0.
\]
Notice that the diagonal coupling \(\Pdstar\) has zero cost, so it is optimal. If a coupling has zero cost, it must satisfy \(P(\{x\neq y\})=0\), so \(P=\Pdstar\). This implies that the unregularised minimiser is unique. Finally, the map \(x\mapsto(x,x)\) is a homeomorphism from \(\Xd\) onto \(\Gammad\). Given that \(\Xd\) is infinite discrete,
it is also noncompact by Lemma~\ref{lem:compact-finite}, and so is \(\Gammad\).
\end{proof}

\begin{proposition}[Total variation convergence]\label{prop:TV}
As \(\varepsilon\downarrow0\), \(\Ped\to \Pdstar\) in total variation.
\end{proposition}
\begin{proof}
Let \(C_0>0\) be the constant from Proposition~\ref{prop:OTunique}, so that \(\Cd(x,y)\ge C_0\,\mathbf{1}_{\{x\neq y\}}\) for all \(x,y\in\Xd\). By optimality of \(\Ped\) and Lemma~\ref{lem:Hmu}, we obtain
\[
\sum_{x,y} \Cd(x,y)\Ped(x,y)
\le J_\varepsilon^{\mathrm d}(\Ped)\le J_\varepsilon^{\mathrm d}(\Pdstar)=\varepsilon H(\mud).
\]
It follows that
\[
C_0\,\Ped(\{x\neq y\})\le \varepsilon H(\mud),
\]
so \(\Ped(\{x\neq y\})\to0\) as \(\varepsilon\downarrow0\). For each \(x\in\Xd\), the marginal constraint gives
\[
\Ped(x,x)\le \sum_y \Ped(x,y)=\mud(x),
\]
and so,
\[
\big|\Ped(x,x)-\mud(x)\big|=\mud(x)-\Ped(x,x).
\]
Finally, we have
\begin{align*}
\|\Ped-\Pdstar\|_{\TV}
&=\frac12\sum_{x\neq y}\Ped(x,y)
  +\frac12\sum_x(\mud(x)-\Ped(x,x))\\
&=\Ped(\{x\neq y\})\to0,
\end{align*}
because \(\sum_x(\mud(x)-\Ped(x,x))=\sum_{x\neq y}\Ped(x,y)\) by the marginal constraint.
\end{proof}

\subsection{\texorpdfstring{The BGN rate on \(\Fd\)}{The BGN rate on Fd}}\label{subsec:discrete-rate}

From Proposition~\ref{prop:TV}, we obtain \(\Pdstar=\lim_{\varepsilon\downarrow0}\Ped\), so \(\Gammad=\supp(\Pdstar)\) and the BGN rate is \(\Id\). We need only the following lower bound.

\begin{lemma}[Uniform positivity of \(I_{\Gamma^{\mathrm d}}\) on \(\Fd\)]\label{lem:I-lower}
For every \((x,y)\in \Fd\), \(\Id(x,y)\ge b-2a\) holds. In particular, \(\inf_{\Fd} \Id\ge b-2a>0\).
\end{lemma}
\begin{proof}
We fix \((x,y)=((n,i),(n,j))\in \Fd\) with \(i\neq j\). In Lemma~\ref{lem:I-two-point}, applied to \((X,Y,C,\Gamma)=(\Xd,\Xd,\Cd,\Gammad)\), we choose
\[
(x',y')=((n,0),(n,0))\in\Gammad.
\]
By \eqref{eq:cost}, we have \(\Cd(x,y)=b\), \(\Cd(x,(n,0))=a\), \(\Cd((n,0),y)=a\), and \(\Cd((n,0),(n,0))=0\). The two-point expression is \(b-2a\), so \(\Id(x,y)\ge b-2a\). Positivity follows immediately from \eqref{eq:params}.
\end{proof}
Moreover, for every \((x,y)\in \Fd\), we also have \(I_{\Gamma^{\mathrm d}}(x,y)\le \Cd(x,y)=b\), because in the definition \eqref{eq:BGNrate-prelim} all terms \(\Cd(x_i,y_i)\) with \(i\ge2\) vanish on \(\Gamma^{\mathrm d}\), but the permuted costs are nonnegative. We conclude that
\[
b-2a\le \inf_{\Fd} I_{\Gamma^{\mathrm d}}\le b<\infty.
\]

\subsection{Proof of the main discrete result}\label{subsec:discrete-proof}

\begin{proof}[Proof of Theorem~\ref{thm:main-discrete}]
The assertion in \ref{item:main-discrete-tv} follows directly from Proposition~\ref{prop:OTunique} and Proposition~\ref{prop:TV}. For \ref{item:main-discrete-closed-set}, by Lemma~\ref{lem:I-lower}, we obtain
\[
\inf_{\Fd} I_{\Gamma^{\mathrm d}}\ge b-2a>0.
\]
Notice that every compact subset of \(\Xd\times\Xd\) is finite by Lemma~\ref{lem:compact-finite}, so the compact sets \(F_n\subset\Fd\) escape every compact set. The estimate in \ref{item:block-estimate-mass} yields, for \(\varepsilon_n:=1/n\),
\[
P^{\mathrm d}_{\varepsilon_n}(F_n)\ge \frac{w_n}{16}\,e^{-(b-\kappa)n}
\]
for all sufficiently large \(n\). We have \(n^{-1}\log(w_n/16)\to0\) by Lemma~\ref{lem:Hmu-proof-app}, so
\[
\limsup_{n\to\infty}\varepsilon_n\log P^{\mathrm d}_{\varepsilon_n}(F_n)\ge -(b-\kappa).
\]
By Corollary~\ref{cor:BGN-full-limit}, we obtain the compact-set upper bound with rate \(I_{\Gamma^{\mathrm d}}\). We now apply Proposition~\ref{prop:escaping-obstruction} to \(F=\Fd\), \(E_n=F_n\), and \(s=b-\kappa\), from which
\[
\beta(\Fd)\ge -(b-\kappa)
\]
follows. Because \(b-\kappa<b-2a\le \inf_{\Fd} I_{\Gamma^{\mathrm d}}\), Proposition~\ref{prop:tail-criterion} shows that
\[
\limsup_{\varepsilon\downarrow0}\varepsilon\log P^{\mathrm d}_\varepsilon(\Fd)>-\inf_{\Fd} I_{\Gamma^{\mathrm d}}.
\]
\end{proof}

\section{A nonatomic lift of the discrete model}\label{sec:nonatomic}
We now lift the discrete model to a nonatomic Polish space. To this end, we replace each label by a compact interval that is clopen in the relative topology of \(X\), and we take the constant cost on label rectangles. This procedure allows us to reduce the continuous Schr\"odinger problem exactly to the label problem.

\subsection{Intervals in the discrete model}

We reuse the sequences $(m_n)_{n\ge1}$, $(L_n)_{n\ge1}$, and $(w_n)_{n\ge1}$ and the parameters $0<2a<\kappa<b$ from Section~\ref{sec:discrete}.
Let $(M_n)_{n\ge1}$ be defined recursively by
\[
M_1:=0,\qquad M_{n+1}:=M_n+3(m_n+1).
\]
For each $n\ge1$ and $i\in\{0,1,\dots,m_n\}$, define the closed unit interval
\[
I_{n,i}:=[M_n+3i,\ M_n+3i+1]\subset\R.
\]
We set
\[
X:=\bigcup_{n\ge1}\ \bigcup_{i=0}^{m_n} I_{n,i}\subset\R,
\]
equipped with the Euclidean metric restricted to $X$.

\begin{lemma}[Polishness and clopen components]\label{lem:X-Polish-clopen-cont}
$X$ is a closed subset of $\R$ and thus Polish. Moreover, each $I_{n,i}$ is clopen in $X$.
\end{lemma}
\begin{proof}
From $m_n\ge2$, we have $M_{n+1}-M_n=3(m_n+1)\ge9$, so $M_n\to+\infty$.
Let $(x_k)\subset X$ converge in $\R$ to $x$. It follows that $(x_k)$ is bounded, say \(x_k\in[-R,R]\) for all \(k\). Let us choose \(N\) with \(M_N>R\). For \(n\ge N\) and an arbitrary \(i\), the interval \(I_{n,i}\subset[M_n,\infty)\subset(R,\infty)\), so \(I_{n,i}\cap[-R,R]=\emptyset\). We infer that \([-R,R]\) meets only finitely many intervals \(I_{n,i}\). We now pass to a subsequence, and assume that \(x_k\in I_{n,i}\) for a fixed pair \((n,i)\). Given that \(I_{n,i}\) is closed in \(\R\), it follows that \(x\in I_{n,i}\subset X\). We conclude that \(X\) is closed in \(\R\), and thus Polish.

Let us now fix \((n,i)\). For any other \((m,j)\neq(n,i)\), the construction gives
\[
\dist(I_{n,i},I_{m,j})\ge2.
\]
Indeed, within a fixed block, consecutive intervals have left endpoints that are separated by \(3\) and length \(1\), so the gap is \(2\). Between successive blocks, the last interval in block \(n\) ends at \(M_n+3m_n+1=M_{n+1}-2\), but the first interval in block \(n+1\) begins at \(M_{n+1}\), so the gap is again \(2\). It follows that \(\dist(I_{n,i},X\setminus I_{n,i})\ge2\), and the open interval
\[
U:=(M_n+3i-1,\ M_n+3i+2)
\]
satisfies \(U\cap X=I_{n,i}\). We deduce that \(I_{n,i}\) is open in \(X\). It is also closed in \(X\), since \(I_{n,i}=X\cap I_{n,i}\) and \(I_{n,i}\) is closed in \(\R\).
\end{proof}

Let $\mathcal{L}^1$ denote Lebesgue measure on $\R$. For each \((n,i)\), let \(\lambda_{n,i}\) denote the Lebesgue measure restricted to \(I_{n,i}\). Since \(|I_{n,i}|=1\), this is a probability measure on \(I_{n,i}\), hence nonatomic. We define $\mu\in\cP(X)$ by
\begin{equation}\label{eq:mu-cont}
\mu:=\sum_{n\ge1}\ \sum_{i=0}^{m_n} \frac{w_n}{m_n+1}\,\lambda_{n,i}.
\end{equation}
It follows that $\mu(I_{n,i})=w_n/(m_n+1)$ and $\sum_n w_n=1$. In particular, observe that $\mu$ is nonatomic. We also define the label map \(\ell:X\to\Xd\) by \(\ell(x)=(n,i)\) if and only if \(x\in I_{n,i}\). For \(\alpha=(n,i)\in\Xd\), we write \(I_\alpha:=I_{n,i}\) and \(\lambda_\alpha:=\lambda_{n,i}\). We then have \(\ell_\#\mu=\mud\). Let us now define
\begin{equation}\label{eq:cost-cont}
C(x,y):=\Cd(\ell(x),\ell(y)),\qquad x,y\in X.
\end{equation}
Because each \(I_\alpha\) is clopen in \(X\) by Lemma~\ref{lem:X-Polish-clopen-cont}, every rectangle \(I_\alpha\times I_\beta\) is clopen in \(X\times X\), and \(C\) is constant on each such rectangle. In particular, \(C\) is continuous.

\subsection{Reduction to the label problem}

In this part, \(\Lambda:=\mu\otimes\mu\), and \(J_\varepsilon\) denotes the Schr\"odinger functional on \(\Pi(\mu,\mu)\) with cost \(C\). We define the block-diagonal coupling as follows:
\begin{equation}\label{eq:Pstar-cont}
P^\star:=\sum_{\alpha\in\Xd} \mud(\alpha)\,(\lambda_\alpha\otimes\lambda_\alpha)\ \in\cP(X\times X).
\end{equation}

\begin{lemma}[Finite entropy and optimality of $P^\star$]\label{lem:Pstar-finite-cont}
\(P^\star\in\Pi(\mu,\mu)\), \(\int C\,dP^\star=0\), so \(P^\star\) is an optimal transport plan for \((\mu,\mu,C)\). Moreover, we have
\[
\KL(P^\star\|\Lambda)=\sum_{\alpha\in\Xd} \mud(\alpha)\log\frac{1}{\mud(\alpha)}<\infty.
\]
In particular, \(J_\varepsilon(P^\star)<\infty\) for all \(\varepsilon>0\), so the continuous Schr\"odinger problem is well posed.
\end{lemma}
\begin{proof}
The marginal identities \((\mathrm{proj}_1)_\#P^\star=(\mathrm{proj}_2)_\#P^\star=\mu\) follow immediately from \eqref{eq:mu-cont} and \eqref{eq:Pstar-cont}, so \(P^\star\in\Pi(\mu,\mu)\). Since \(C\ge0\) on \(X\times X\) and vanishes whenever \(\ell(x)=\ell(y)\), we also have \(\int C\,dP^\star=0\). We conclude that \(P^\star\) is optimal for the unregularised problem.

For $\alpha\in\Xd$, we set $R_\alpha:=I_\alpha\times I_\alpha$. The rectangles $(R_\alpha)_{\alpha\in\Xd}$ are pairwise disjoint, and by definition of $P^\star$, it follows that
\[
P^\star|_{R_\alpha}=\mud(\alpha)\,(\lambda_\alpha\otimes\lambda_\alpha),
\qquad
P^\star(R_\alpha)=\mud(\alpha).
\]
Moreover, since $\mu|_{I_\alpha}=\mud(\alpha)\lambda_\alpha$, we have
\[
\Lambda|_{R_\alpha}=(\mu\otimes\mu)|_{R_\alpha}
=\mud(\alpha)^2\,(\lambda_\alpha\otimes\lambda_\alpha).
\]
We obtain \(P^\star\ll\Lambda\), and on \(R_\alpha\),
\[
\frac{dP^\star}{d\Lambda}=\frac{1}{\mud(\alpha)}.
\]
Therefore,
\[
\KL(P^\star\|\Lambda)
=\sum_{\alpha\in\Xd}\int_{R_\alpha}\log\Big(\frac{dP^\star}{d\Lambda}\Big)\,dP^\star
=\sum_{\alpha\in\Xd}\mud(\alpha)\log\frac{1}{\mud(\alpha)}.
\]
Note that the series is finite by Lemma~\ref{lem:Hmu-proof-app}, so \(J_\varepsilon(P^\star)<\infty\) for all \(\varepsilon>0\).
\end{proof}

\begin{remark}
Unlike the discrete label problem, the unregularised transport problem for \((\mu,\mu,C)\) is not unique. Indeed, observe that for an arbitrary family \((\rho_\alpha)_{\alpha\in\Xd}\) with \(\rho_\alpha\in\Pi(\lambda_\alpha,\lambda_\alpha)\),
\[
\widetilde P:=\sum_{\alpha\in\Xd}\mud(\alpha)\rho_\alpha
\]
belongs to \(\Pi(\mu,\mu)\) and satisfies \(\int C\,d\widetilde P=0\). The important point here is that the entropic minimisers select the particular block-product plan \(P^\star\).
\end{remark}

Among zero-cost couplings, only the relative-entropy term makes a distinction between competitors. In the next lemma, we show the entropy chain rule along the partition \((I_\alpha\times I_\beta)_{\alpha,\beta\in\Xd}\), and then construct Proposition~\ref{prop:coarse-principle} to reduce the continuous problem exactly to the discrete label model.

Let \(\widehat\Pi(\mud,\mud)\) denote the set of Borel probability measures on \(\Xd\times\Xd\) with both marginals equal to \(\mud\). We define the discrete entropic functional on $\widehat\Pi(\mud,\mud)$ by
\[
\widehat J_\varepsilon(\widehat P)
:=\sum_{\alpha,\beta\in\Xd} \Cd(\alpha,\beta)\,\widehat P(\alpha,\beta)
+\varepsilon\,\KL(\widehat P\|\Lambda^{\mathrm d}).
\]
Let \(\widehat P_\varepsilon\) denote the unique minimiser of \(\widehat J_\varepsilon\). Note that existence and uniqueness follow directly from Lemma~\ref{lem:Hmu} and Proposition~\ref{prop:exist-unique}. Because the label problem \((\Xd,\mud,\mud,\Cd)\) is precisely the discrete model from Section~\ref{sec:discrete}, we identify
\[
\widehat P_\varepsilon=P^{\mathrm d}_\varepsilon,\qquad \varepsilon>0,
\]
and apply the estimates we have already established. We write
\begin{equation}\label{eq:Phat-star}
\widehat P^\star(\alpha,\beta):=\mud(\alpha)\mathbf{1}_{\{\alpha=\beta\}},
\qquad \alpha,\beta\in\Xd,
\end{equation}
for the diagonal coupling on the label space. Given $\widehat P\in\widehat\Pi(\mud,\mud)$, we define its lift $\mathcal L(\widehat P)\in\Pi(\mu,\mu)$ by
\begin{equation}\label{eq:lift-cont}
\mathcal L(\widehat P):=\sum_{\alpha,\beta\in\Xd} \widehat P(\alpha,\beta)\,(\lambda_\alpha\otimes\lambda_\beta).
\end{equation}
For a given $P\in\Pi(\mu,\mu)$, we define its coarse-graining $\widehat P\in\widehat\Pi(\mud,\mud)$ by
\begin{equation}\label{eq:pushforward-cont}
\widehat P:= (\ell\times\ell)_\# P,
\qquad\text{i.e.}\qquad
\widehat P(\alpha,\beta)=P(I_\alpha\times I_\beta).
\end{equation}

\begin{lemma}[Coarse-graining of cost and relative entropy]\label{lem:KL-chain-cont}
Let $P\in\Pi(\mu,\mu)$ and let $\widehat P=(\ell\times\ell)_\#P$. It follows that
\[
\int_{X\times X} C\,dP=\sum_{\alpha,\beta\in\Xd} \Cd(\alpha,\beta)\,\widehat P(\alpha,\beta).
\]
Moreover, $\KL(P\|\Lambda)\ge \KL(\widehat P\|\Lambda^{\mathrm d})$.
If $\KL(P\|\Lambda)<\infty$, then $P\ll\Lambda$ and
\[
\KL(P\|\Lambda)
=
\KL(\widehat P\|\Lambda^{\mathrm d})
+\sum_{\substack{\alpha,\beta\in\Xd\\ \widehat P(\alpha,\beta)>0}}
\widehat P(\alpha,\beta)\,
\KL\!\big(P^{\alpha\beta}\,\big\|\,\lambda_\alpha\otimes\lambda_\beta\big),
\]
where, for each \((\alpha,\beta)\) with \(\widehat P(\alpha,\beta)>0\),
\(P^{\alpha\beta}\) denotes the conditional probability of \(P\) on
\(I_\alpha\times I_\beta\). In that case, equality
\(\KL(P\|\Lambda)=\KL(\widehat P\|\Lambda^{\mathrm d})\) holds if and only if
\(P^{\alpha\beta}=\lambda_\alpha\otimes\lambda_\beta\) for every
\((\alpha,\beta)\) with \(\widehat P(\alpha,\beta)>0\).
\end{lemma}
\begin{proof}
Let \(R_{\alpha\beta}:=I_\alpha\times I_\beta\). Because \(C\) is constant on each \(R_{\alpha\beta}\), we have
\[
\int_{X\times X} C\,dP
=\sum_{\alpha,\beta\in\Xd}\Cd(\alpha,\beta)\,P(R_{\alpha\beta})
=\sum_{\alpha,\beta\in\Xd}\Cd(\alpha,\beta)\,\widehat P(\alpha,\beta).
\]
Moreover, since \(\widehat P=(\ell\times\ell)_\#P\) and \(\Lambda^{\mathrm d}=(\ell\times\ell)_\#\Lambda\), the inequality gives
\[
\KL(P\|\Lambda)\ge \KL(\widehat P\|\Lambda^{\mathrm d}).
\]
Let us now assume that \(\KL(P\|\Lambda)<\infty\), from which \(P\ll\Lambda\) follows.
For each \((\alpha,\beta)\) with \(\widehat P(\alpha,\beta)>0\), let
\[
p_{\alpha\beta}
:=\frac{dP^{\alpha\beta}}{d(\lambda_\alpha\otimes\lambda_\beta)}.
\]
Since
\[
P|_{R_{\alpha\beta}}
=\widehat P(\alpha,\beta)\,P^{\alpha\beta},
\qquad
\Lambda|_{R_{\alpha\beta}}
=\mud(\alpha)\mud(\beta)\,(\lambda_\alpha\otimes\lambda_\beta),
\]
we have, on \(R_{\alpha\beta}\),
\[
\frac{dP}{d\Lambda}
=
\frac{\widehat P(\alpha,\beta)}{\mud(\alpha)\mud(\beta)}\,p_{\alpha\beta}
\qquad \Lambda\text{-a.e.}
\]
Finally, we obtain
\begin{align*}
\KL(P\|\Lambda)
&=
\sum_{\substack{\alpha,\beta\in\Xd\\ \widehat P(\alpha,\beta)>0}}
\int_{R_{\alpha\beta}}
\log\!\Big(
\frac{\widehat P(\alpha,\beta)}{\mud(\alpha)\mud(\beta)}\,p_{\alpha\beta}
\Big)\,dP \\
&=
\sum_{\alpha,\beta\in\Xd}
\widehat P(\alpha,\beta)\log\frac{\widehat P(\alpha,\beta)}{\mud(\alpha)\mud(\beta)}
\\
&\quad+
\sum_{\substack{\alpha,\beta\in\Xd\\ \widehat P(\alpha,\beta)>0}}
\widehat P(\alpha,\beta)
\int \log p_{\alpha\beta}\,dP^{\alpha\beta}.
\end{align*}
Notice that the first term is \(\KL(\widehat P\|\Lambda^{\mathrm d})\), and the second term is
\[
\sum_{\substack{\alpha,\beta\in\Xd\\ \widehat P(\alpha,\beta)>0}}
\widehat P(\alpha,\beta)\,
\KL\!\big(P^{\alpha\beta}\,\big\|\,\lambda_\alpha\otimes\lambda_\beta\big).
\]
This proves the decomposition. The equality
\(\KL(P\|\Lambda)=\KL(\widehat P\|\Lambda^{\mathrm d})\)
holds if and only if each conditional relative entropy vanishes, that is, if and only if
\(P^{\alpha\beta}=\lambda_\alpha\otimes\lambda_\beta\)
for every \((\alpha,\beta)\) with \(\widehat P(\alpha,\beta)>0\).
\end{proof}

\begin{proposition}[Exact reduction to the label problem]\label{prop:coarse-principle}
In the partitioned case above (clopen atoms \((I_\alpha)_{\alpha\in\Xd}\), conditionally uniform \(\mu\), and block-constant cost \(C\) induced by \(\Cd\)), the following hold for every \(\varepsilon>0\).
\begin{enumerate}[label=\textup{(\roman*)},ref=\textup{(\roman*)},leftmargin=2.2em]
\item\label{item:coarse-principle-minimizers}
\[
\inf_{P\in\Pi(\mu,\mu)} J_\varepsilon(P)
=
\inf_{\widehat P\in\widehat\Pi(\mud,\mud)} \widehat J_\varepsilon(\widehat P),
\]
and
\[
\operatorname*{argmin}_{P\in\Pi(\mu,\mu)} J_\varepsilon(P)
=
\Bigl\{\mathcal L(\widehat P): \widehat P\in\operatorname*{argmin}_{\widehat P\in\widehat\Pi(\mud,\mud)} \widehat J_\varepsilon(\widehat P)\Bigr\}.
\]
As a result, the unique minimiser of \(J_\varepsilon\) is \(P_\varepsilon=\mathcal L(\widehat P_\varepsilon)\).

\item\label{item:coarse-principle-rectangles}
For each \((\alpha,\beta)\in\Xd\times\Xd\),
\[
P_\varepsilon|_{I_\alpha\times I_\beta}
=
\widehat P_\varepsilon(\alpha,\beta)\,(\lambda_\alpha\otimes\lambda_\beta).
\]
It follows that \(P_\varepsilon(I_\alpha\times I_\beta)=\widehat P_\varepsilon(\alpha,\beta)\), and whenever \(\widehat P_\varepsilon(\alpha,\beta)>0\), the conditional law of \(P_\varepsilon\) on \(I_\alpha\times I_\beta\) is \(\lambda_\alpha\otimes\lambda_\beta\). In particular, if \(\widehat A\subset\Xd\times\Xd\) and
\[
A:=\bigcup_{(\alpha,\beta)\in\widehat A}(I_\alpha\times I_\beta),
\]
then \(P_\varepsilon(A)=\widehat P_\varepsilon(\widehat A)\).

\item\label{item:coarse-principle-tv}
For an arbitrary \(\widehat P,\widehat Q\in\widehat\Pi(\mud,\mud)\),
\[
\|\mathcal L(\widehat P)-\mathcal L(\widehat Q)\|_{\TV}
=
\|\widehat P-\widehat Q\|_{\TV}.
\]
\end{enumerate}
\end{proposition}
\begin{proof}
For \(P\in\Pi(\mu,\mu)\), let \(\widehat P:=(\ell\times\ell)_\#P\). Lemma~\ref{lem:KL-chain-cont} gives
\[
J_\varepsilon(P)\ge \widehat J_\varepsilon(\widehat P),
\]
so
\[
\inf_{P\in\Pi(\mu,\mu)}J_\varepsilon(P)\ge \inf_{\widehat P\in\widehat\Pi(\mud,\mud)}\widehat J_\varepsilon(\widehat P).
\]

To show the other direction, we fix \(\widehat Q\in\widehat\Pi(\mud,\mud)\). Notice that its lift \(\mathcal L(\widehat Q)\) belongs to \(\Pi(\mu,\mu)\), since
\[
(\mathrm{proj}_1)_\#\mathcal L(\widehat Q)
=\sum_{\alpha}\Bigl(\sum_\beta \widehat Q(\alpha,\beta)\Bigr)\lambda_\alpha
=\sum_\alpha \mud(\alpha)\lambda_\alpha
=\mu,
\]
and, analogously, for the second marginal. The coarse-graining of \(\mathcal L(\widehat Q)\) is exactly \(\widehat Q\), and on each rectangle \(R_{\alpha\beta}:=I_\alpha\times I_\beta\) with \(\widehat Q(\alpha,\beta)>0\) the conditional law of \(\mathcal L(\widehat Q)\) is \(\lambda_\alpha\otimes\lambda_\beta\). We deduce that equality holds in Lemma~\ref{lem:KL-chain-cont}, so
\[
J_\varepsilon(\mathcal L(\widehat Q))=\widehat J_\varepsilon(\widehat Q).
\]
We now take infima over \(\widehat Q\) to obtain the reverse inequality and thus equality of the variational values.

Now suppose that \(\widehat Q\) minimises \(\widehat J_\varepsilon\). It follows that \(\mathcal L(\widehat Q)\) minimises \(J_\varepsilon\). On the other hand, if \(P\) minimises \(J_\varepsilon\) and \(\widehat P:=(\ell\times\ell)_\#P\), then \(\widehat P\) minimises \(\widehat J_\varepsilon\), and equality holds in Lemma~\ref{lem:KL-chain-cont}. We conclude that \(P=\mathcal L(\widehat P)\). This proves the identity of argmin sets.

Given that \(\widehat J_\varepsilon\) has the unique minimiser \(\widehat P_\varepsilon\), the continuous problem has the unique minimiser
\[
P_\varepsilon=\mathcal L(\widehat P_\varepsilon).
\]
The formula on each rectangle is \eqref{eq:lift-cont}. Finally, because the rectangles \(I_\alpha\times I_\beta\) are pairwise disjoint, we have
\[
\|\mathcal L(\widehat P)-\mathcal L(\widehat Q)\|_{\TV}
=
\frac12\sum_{\alpha,\beta\in\Xd}
\big|\widehat P(\alpha,\beta)-\widehat Q(\alpha,\beta)\big|
=
\|\widehat P-\widehat Q\|_{\TV}.
\]
The statements in \ref{item:coarse-principle-rectangles} and \ref{item:coarse-principle-tv} now follow immediately.
\end{proof}

For the rest of Section~\ref{sec:nonatomic}, we use hats to denote quantities for label levels. For the sequence \(\varepsilon_n:=1/n\), we set
\[
\widehat P_n:=\widehat P_{\varepsilon_n},
\qquad
P_n:=P_{\varepsilon_n}=\mathcal L(\widehat P_{\varepsilon_n}),
\qquad n\ge1.
\]

\subsection{Transfer of the discrete case}

Recall the diagonal label coupling \(\widehat P^\star\) from \eqref{eq:Phat-star}. By \eqref{eq:Pstar-cont} and \eqref{eq:lift-cont}, we have
\[
P^\star=\mathcal L(\widehat P^\star).
\]
For each \(n\ge1\), let
\[
\begin{aligned}
\widehat F_n
&:=\{((n,i),(n,j)):1\le i\neq j\le m_n\}\subset\Xd\times\Xd,\\
\mathcal F_n
&:=(\ell\times\ell)^{-1}(\widehat F_n)\subset X\times X,
\end{aligned}
\]
and define
\begin{equation}\label{eq:F-cont}
F:=\bigcup_{n\ge1}\mathcal F_n=(\ell\times\ell)^{-1}(\Fd)\subset X\times X.
\end{equation}

\begin{proposition}[Transfer from the label model]\label{prop:transfer-cont}
The following properties are true.
\begin{enumerate}[label=\textup{(\roman*)},ref=\textup{(\roman*)},leftmargin=2.2em]
\item\label{item:transfer-cont-tv}
\(P_\varepsilon\to P^\star\) in total variation, and
\[
\Gamma:=\supp(P^\star)=(\ell\times\ell)^{-1}(\Gammad)
=\bigcup_{\alpha\in\Xd}(I_\alpha\times I_\alpha).
\]
In particular, \(\Gamma\) is noncompact.

\item\label{item:transfer-cont-set}
\(F\) is clopen and noncompact, and for every \(\varepsilon>0\), it holds that
\[
P_\varepsilon(F)=\widehat P_\varepsilon(\Fd)=P_\varepsilon^{\mathrm d}(\Fd).
\]
More generally, we have \(P_\varepsilon(\mathcal F_n)=\widehat P_\varepsilon(\widehat F_n)\) for every \(n\ge1\).

\item\label{item:transfer-cont-rate}
For every \((x,y)\in X\times X\),
\[
I_\Gamma(x,y)=I_{\Gamma^{\mathrm d}}(\ell(x),\ell(y)).
\]
As a result, we get
\[
\inf_F I_\Gamma=\inf_{\Fd} I_{\Gamma^{\mathrm d}},\qquad
b-2a\le \inf_F I_\Gamma\le b<\infty.
\]
\end{enumerate}
\end{proposition}

\begin{proof}
For \ref{item:transfer-cont-tv}, Proposition~\ref{prop:TV} gives \(\widehat P_\varepsilon\to\widehat P^\star\) in total variation. The statement in \ref{item:coarse-principle-tv} yields
\[
\|P_\varepsilon-P^\star\|_{\TV}
=
\|\widehat P_\varepsilon-\widehat P^\star\|_{\TV}\to0.
\]
Let us set
\[
G:=(\ell\times\ell)^{-1}(\Gammad)=\bigcup_{\alpha\in\Xd}(I_\alpha\times I_\alpha).
\]
Notice that each rectangle \(I_\alpha\times I_\beta\) is clopen in \(X\times X\), and \(P^\star\) charges exactly the diagonal rectangles. On each \(I_\alpha\times I_\alpha\), the conditional law of \(P^\star\) is \(\lambda_\alpha\otimes\lambda_\alpha\), which has full support. This implies \(\supp(P^\star)=G\). Since \(M_n\to\infty\), the set \(G\subset X\times X\subset\R^2\) is unbounded and thus noncompact.

For \ref{item:transfer-cont-set}, the sets \(I_\alpha\times I_\beta\) create a clopen partition of \(X\times X\), and \(F\) is the union of the subfamily indexed by \(\Fd\). It follows that \(F\) is clopen. Because \(M_n\to\infty\), the set \(F\) is unbounded and thus noncompact. The mass identities follow from \ref{item:coarse-principle-rectangles} and the identification \(\widehat P_\varepsilon=P_\varepsilon^{\mathrm d}\).

For \ref{item:transfer-cont-rate}, we fix \((x,y)\in X\times X\), and set \(\alpha:=\ell(x)\), \(\beta:=\ell(y)\). Let us consider an admissible term in \eqref{eq:BGNrate-prelim} which defines \(I_\Gamma(x,y)\), with \(k\ge2\), \((x_i,y_i)\in\Gamma\) for \(i=2,\dots,k\), and \(\sigma\in\Sigma(k)\). Since \(\Gamma=(\ell\times\ell)^{-1}(\Gammad)\), for each \(i\ge2\), there exists \(\gamma_i\in\Xd\) such that
\[
(\ell(x_i),\ell(y_i))=(\gamma_i,\gamma_i)\in\Gammad.
\]
Because \(C(u,v)=\Cd(\ell(u),\ell(v))\), the respective cycle increment is
\[
\sum_{i=1}^k \Cd(\alpha_i,\beta_i)-\sum_{i=1}^k \Cd(\alpha_i,\beta_{\sigma(i)}),
\]
where \((\alpha_1,\beta_1)=(\alpha,\beta)\) and \((\alpha_i,\beta_i)=(\gamma_i,\gamma_i)\) for \(i\ge2\). Note that this is an admissible discrete increment in the definition of \(I_{\Gamma^{\mathrm d}}(\alpha,\beta)\). We conclude that
\[
I_\Gamma(x,y)\le I_{\Gamma^{\mathrm d}}(\ell(x),\ell(y)).
\]
On the other hand, every admissible discrete choice can be realised by selecting arbitrary representatives in the corresponding intervals \(I_{\gamma_i}\). Therefore,
\[
I_\Gamma(x,y)\ge I_{\Gamma^{\mathrm d}}(\ell(x),\ell(y)).
\]
This establishes the equality. We now take the infimum over \(F=(\ell\times\ell)^{-1}(\Fd)\) and obtain
\[
\inf_F I_\Gamma=\inf_{\Fd} I_{\Gamma^{\mathrm d}},
\]
For the upper bound, we fix an arbitrary \((x,y)\in F\). We then have \((\ell(x),\ell(y))\in \Fd\), so \(I_{\Gamma^{\mathrm d}}(\ell(x),\ell(y))\le b\) by the one-point bound given after Lemma~\ref{lem:I-lower}. We conclude that \(\inf_F I_\Gamma\le b\). The lower bound now follows immediately from Lemma~\ref{lem:I-lower}.
\end{proof}

\begin{proof}[Proof of Theorem~\ref{thm:main-cont}]
The statement in \ref{item:transfer-cont-tv} proves \ref{item:main-cont-tv}.
By the results from \ref{item:transfer-cont-set}, we have
\[
P_\varepsilon(F)=P_\varepsilon^{\mathrm d}(\Fd),\qquad \varepsilon>0.
\]
Therefore, \ref{item:main-discrete-closed-set} gives
\[
\limsup_{\varepsilon\downarrow0}\varepsilon\log P_\varepsilon(F)
=
\limsup_{\varepsilon\downarrow0}\varepsilon\log P_\varepsilon^{\mathrm d}(\Fd)
>
-\inf_{\Fd} I_{\Gamma^{\mathrm d}}.
\]
By \ref{item:transfer-cont-rate}, we get
\[
\inf_F I_\Gamma=\inf_{\Fd} I_{\Gamma^{\mathrm d}},
\]
so
\[
\limsup_{\varepsilon\downarrow0}\varepsilon\log P_\varepsilon(F)>-\inf_F I_\Gamma.
\]
\end{proof}

\subsection{\texorpdfstring{No full LDP on \(X\times X\)}{No full LDP on X x X}}\label{subsec:no-ambient-rate}

We say that \((P_\varepsilon)\) satisfies a full LDP on \(X\times X\) with rate \(I_{\mathrm{amb}}\) if \(I_{\mathrm{amb}}:X\times X\to[0,\infty]\) is lower semicontinuous and
\[
\limsup_{\varepsilon\downarrow0}\varepsilon\log P_\varepsilon(\mathcal C)\le -\inf_{\mathcal C} I_{\mathrm{amb}}
\qquad\text{for every closed }\mathcal C\subset X\times X,
\]
\[
\liminf_{\varepsilon\downarrow0}\varepsilon\log P_\varepsilon(O)\ge -\inf_O I_{\mathrm{amb}}
\qquad\text{for every open }O\subset X\times X.
\]

\begin{lemma}[Decay exponent on a fixed off-diagonal rectangle]\label{lem:fixed-rectangle-exponent}
We fix \(N\ge1\) and \(1\le i\neq j\le m_N\). Let us also set
\[
R_{N,i,j}:=I_{N,i}\times I_{N,j}\subset X\times X.
\]
It follows that \(R_{N,i,j}\) is clopen and
\[
\lim_{\varepsilon\downarrow0}\varepsilon\log P_\varepsilon(R_{N,i,j})=-b.
\]
\end{lemma}

\begin{proof}
By Lemma~\ref{lem:X-Polish-clopen-cont}, we know that \(R_{N,i,j}\) is clopen. Proposition~\ref{prop:coarse-principle} gives
\[
P_\varepsilon(R_{N,i,j})
=
\widehat P_\varepsilon\big(((N,i),(N,j))\big).
\]
For a fixed \(N\), the permutations of \(\{1,\dots,m_N\}\) that act inside block \(N\) preserve \(\mud\) and \(\Cd\). We now apply Lemma~\ref{lem:symmetry} to the label problem and get
\[
\begin{aligned}
\widehat P_\varepsilon(((N,r),(N,r)))
&=
\widehat P_\varepsilon(((N,1),(N,1))),\\
\widehat P_\varepsilon(((N,r),(N,s)))
&=
\widehat P_\varepsilon(((N,1),(N,2)))
\end{aligned}
\]
for all \(r,s\in\{1,\dots,m_N\}\) with \(r\neq s\).
It follows that
\[
\widehat P_\varepsilon\big(((N,i),(N,j))\big)
=
\widehat P_\varepsilon\big(((N,1),(N,2))\big)
=
\widehat P_\varepsilon\big(((N,1),(N,1))\big)e^{-b/\varepsilon}
\]
by Lemma~\ref{lem:o-vs-d-app}. By Proposition~\ref{prop:TV},
\[
\widehat P_\varepsilon\big(((N,1),(N,1))\big)\to \mud((N,1))>0,
\]
and so \(\varepsilon\log \widehat P_\varepsilon(((N,1),(N,1)))\to0\). As a result,
\[
\varepsilon\log P_\varepsilon(R_{N,i,j})
=
\varepsilon\log \widehat P_\varepsilon\big(((N,1),(N,1))\big)-b
\longrightarrow -b.
\]
\end{proof}

\begin{proof}[Proof of Theorem~\ref{thm:intro-no-ambient}]
Let us assume, for the sake of contradiction, that \((P_\varepsilon)\) satisfies a full LDP on \(X\times X\) with lower semicontinuous rate \(I_{\mathrm{amb}}\). We fix \(N\ge1\) and \(1\le i\neq j\le m_N\). Given that \(R_{N,i,j}\) is both open and closed, the large-deviation upper and lower bounds imply
\[
\lim_{\varepsilon\downarrow0}\varepsilon\log P_\varepsilon(R_{N,i,j})
=-\inf_{R_{N,i,j}}I_{\mathrm{amb}}.
\]
By Lemma~\ref{lem:fixed-rectangle-exponent}, the left-hand side equals \(-b\). We have
\[
\inf_{R_{N,i,j}}I_{\mathrm{amb}}=b
\qquad\text{for every }N\ge1,\ 1\le i\neq j\le m_N.
\]
Since
\[
F=\bigcup_{n\ge1}\ \bigcup_{\substack{1\le i,j\le m_n\\ i\neq j}} R_{n,i,j},
\]
we obtain
\[
\inf_F I_{\mathrm{amb}}
=
\inf_{n\ge1}\ \inf_{\substack{1\le i,j\le m_n\\ i\neq j}}\ \inf_{R_{n,i,j}}I_{\mathrm{amb}}
=
b.
\]
Because \(F\) is closed, the large-deviation upper bound gives
\[
\limsup_{\varepsilon\downarrow0}\varepsilon\log P_\varepsilon(F)\le -b.
\]
On the other hand, \(\mathcal F_n\subset F\) holds for every \(n\). By \ref{item:transfer-cont-set} and \ref{item:block-estimate-mass}, we have:
\[
P_{\varepsilon_n}(\mathcal F_n)=\widehat P_{\varepsilon_n}(\widehat F_n)
\ge \frac{w_n}{16}\,e^{-(b-\kappa)n}
\]
for all sufficiently large \(n\), where \(\varepsilon_n:=1/n\). From \(n^{-1}\log(w_n/16)\to0\), we deduce that
\[
\limsup_{n\to\infty}\varepsilon_n\log P_{\varepsilon_n}(\mathcal F_n)\ge -(b-\kappa).
\]
We also note that
\[
\limsup_{\varepsilon\downarrow0}\varepsilon\log P_\varepsilon(F)\ge -(b-\kappa),
\]
but this contradicts \(\limsup_{\varepsilon\downarrow0}\varepsilon\log P_\varepsilon(F)\le -b\). We conclude that no such function \(I_{\mathrm{amb}}\) exists.
\end{proof}

\appendix
\section{Criterion for closed-set upper bounds}\label{app:superlevel-tail}
In this part, we provide a convenient sufficient condition, in terms of \(I_\Gamma\), under which the compact-set upper bound extends to all closed sets. In our counterexample, the obstruction comes from the tail exponent on the escaping closed set \(F\), that is,
\[
\beta(F)\ge -(b-\kappa)>-\inf_F I_\Gamma.
\]
Below, we identify the condition that rules out this mechanism.

\subsection{\texorpdfstring{A superlevel-tail condition}{A superlevel-tail condition}}\label{subsec:tail-away0}

We assume the following compact-set upper bound, and then show a sufficient condition for extending this bound from compact sets to closed sets.
\begin{equation}\label{eq:BGN-compact-add}
\limsup_{\varepsilon\downarrow0}\ \varepsilon\log P_{\varepsilon}(\mathcal K)\ \le\ -\inf_{\mathcal K} I_\Gamma
\qquad\text{for every compact }\mathcal K\subset X\times Y.
\end{equation}

\begin{remark}[Strict superlevel sets for lower semicontinuous rates]\label{rem:strict-superlevels}
If \(f\) is lower semicontinuous, then \(\{f>\delta\}\) is open, but \(\{f\ge \delta\}\) need not be closed. For this reason, we analyse the tail conditions with strict superlevels \(\{I_\Gamma>\delta\}\). For the BGN rate, \(I_\Gamma\) is lower semicontinuous when \(C\) is continuous (see \cite[Lem.~4.2]{BGN}).
\end{remark}

\begin{definition}[Superlevel-tail condition for \(I_\Gamma\)]\label{def:I-tight-away0}
We say that \((P_\varepsilon)_{\varepsilon>0}\) satisfies the \emph{superlevel-tail condition for \(I_\Gamma\)} if for every \(\delta>0\),
\begin{equation}\label{eq:I-tight-away0}
\inf_{K\subset X\times Y\ \mathrm{compact}}
\ \limsup_{\varepsilon\downarrow0}\ \varepsilon\log P_{\varepsilon}\big(\{I_\Gamma>\delta\}\cap K^c\big)
\ \le\ -\delta.
\end{equation}
\end{definition}

\begin{proposition}[The superlevel-tail condition implies the closed-set upper bound]\label{prop:I-tight-away0-implies-closed}
Let us assume \eqref{eq:BGN-compact-add}, and suppose that \((P_\varepsilon)\) satisfies the superlevel-tail condition for \(I_\Gamma\) in the sense of Definition~\ref{def:I-tight-away0}. It follows that for every closed set \(F\subset X\times Y\),
\[
\limsup_{\varepsilon\downarrow0}\ \varepsilon\log P_{\varepsilon}(F)\ \le\ -\inf_{F} I_\Gamma.
\]
In the notation of Proposition~\ref{prop:tail-criterion}, this is \(\beta(F)\le -\inf_F I_\Gamma\) for all closed \(F\).
\end{proposition}

\begin{proof}
We fix a closed set \(F\subset X\times Y\) and set \(m:=\inf_F I_\Gamma\in[0,\infty]\). If \(m=0\), then the bound follows automatically since \(P_\varepsilon(F)\le 1\) implies
\(\limsup_{\varepsilon\downarrow0}\varepsilon\log P_\varepsilon(F)\le 0=-m\). If \(m<\infty\), we fix \(\eta\in(0,m)\) and set \(\delta:=m-\eta\). It follows that \(F\subset\{I_\Gamma>\delta\}\), so for every compact \(K\subset X\times Y\),
\[
P_\varepsilon(F\cap K^c)\le P_\varepsilon(\{I_\Gamma>\delta\}\cap K^c).
\]
If we now take \(\varepsilon\log\), then \(\limsup_{\varepsilon\downarrow0}\), and then \(\inf_K\), we arrive at
\[
\beta(F)\le
\inf_{K}\limsup_{\varepsilon\downarrow0}\varepsilon\log P_\varepsilon(\{I_\Gamma>\delta\}\cap K^c)
\le -\delta=-(m-\eta)
\]
by Definition~\ref{def:I-tight-away0}. It now suffices to let \(\eta\downarrow0\) to obtain \(\beta(F)\le -m\). If \(m=\infty\), we fix \(M>0\). We then have \(F\subset\{I_\Gamma>M\}\), and the same argument with \(\delta:=M\) gives \(\beta(F)\le -M\). Note that \(M>0\) is arbitrary, so \(\beta(F)=-\infty\) follows.

Finally, by Proposition~\ref{prop:tail-criterion} (applied with \(\cZ=X\times Y\), \(\cQ_\varepsilon=P_\varepsilon\), and \(I=I_\Gamma\)), we have
\[
\limsup_{\varepsilon\downarrow0}\varepsilon\log P_\varepsilon(F)
\le
\max\big\{-\inf_F I_\Gamma,\ \beta(F)\big\}
=
-\inf_F I_\Gamma,
\]
as required.
\end{proof}

\begin{corollary}[Failure of the superlevel-tail condition]\label{cor:fail-I-tight-away0}
In the context of Theorem~\ref{thm:main-cont}, the family \((P_\varepsilon)_{\varepsilon>0}\) does not satisfy the superlevel-tail condition of Definition~\ref{def:I-tight-away0}. More precisely, for every
\[
\delta\in(b-\kappa,\; b-2a),
\]
\[
\inf_{K\subset X\times X\ \mathrm{compact}}
\limsup_{\varepsilon\downarrow0}\varepsilon\log
P_\varepsilon\bigl(\{I_\Gamma>\delta\}\cap K^c\bigr)
\ge -(b-\kappa)>-\delta.
\]
\end{corollary}

\begin{proof}
By \ref{item:transfer-cont-rate}, we have
\[
\inf_F I_\Gamma=\inf_{\Fd}I_{\Gamma^{\mathrm d}}\ge b-2a>\delta,
\]
so \(F\subset\{I_\Gamma>\delta\}\). This implies that for every compact \(K\subset X\times X\),
\[
P_\varepsilon\bigl(\{I_\Gamma>\delta\}\cap K^c\bigr)\ge P_\varepsilon(F\cap K^c).
\]

Let us now fix such a compact \(K\). From \(M_n\to\infty\), we infer \(\mathcal F_n\subset K^c\) for all sufficiently large \(n\). By \ref{item:transfer-cont-set} and \ref{item:block-estimate-mass}, we obtain
\[
P_{\varepsilon_n}(F\cap K^c)\ge P_{\varepsilon_n}(\mathcal F_n)
=\widehat P_{\varepsilon_n}(\widehat F_n)
\ge \frac{w_n}{16}\,e^{-(b-\kappa)n}
\]
for all sufficiently large \(n\), where \(\varepsilon_n:=1/n\). From \(n^{-1}\log(w_n/16)\to0\), we get
\[
\limsup_{\varepsilon\downarrow0}\varepsilon\log P_\varepsilon(F\cap K^c)\ge -(b-\kappa).
\]
We now take the infimum over \(K\) to obtain \(\beta(F)\ge -(b-\kappa)\), and so:
\[
\inf_{K}\limsup_{\varepsilon\downarrow0}\varepsilon\log
P_\varepsilon\bigl(\{I_\Gamma>\delta\}\cap K^c\bigr)
\ge \beta(F)\ge -(b-\kappa).
\]
The final strict inequality follows immediately from \(\delta>b-\kappa\).
\end{proof}

\begin{remark}[A one-set tail criterion]\label{rem:local-tail-dom}
We fix a closed set \(F\subset\cZ\). If there exist compact sets \((K_R)_{R\ge1}\) and numbers \((\psi_R)_{R\ge1}\) with \(\psi_R\uparrow\infty\) such that
\[
\limsup_{\varepsilon\downarrow0}\varepsilon\log\cQ_\varepsilon(F\cap K_R^c)\le -\psi_R
\qquad\text{for every }R\ge1,
\]
then \(\beta(F)=-\infty\). By Proposition~\ref{prop:tail-criterion}, we have
\[
\limsup_{\varepsilon\downarrow0}\varepsilon\log\cQ_\varepsilon(F)\le -\inf_{z\in F}I(z).
\]
\end{remark}

In Remark~\ref{rem:local-tail-dom}, we gave a convenient one-set criterion for \(\beta(F)=-\infty\). Note that the abstract tail mechanism is standard, so the remaining question, specific to EOT, is to identify the intrinsic hypotheses on \((X,Y,C)\), or on the limiting support \(\Gamma\), which guarantee \(\beta(F)\le -\inf_F I_\Gamma\) for every closed \(F\) without using exponential tightness. Such a hypothesis must exclude the multiplicity-growth mechanism, possibly by applying coercivity at infinity or even geometric restrictions on near-optimal cycles.

\section{Schr\"odinger facts on countable discrete spaces}\label{app:schroedinger}
We present here selected auxiliary facts on the discrete Schr\"odinger problem that we applied in Section~\ref{sec:discrete}.

\begin{proposition}[Schr\"odinger structure]\label{prop:exist-unique}
For every \(\varepsilon>0\), \(J_\varepsilon^{\mathrm d}\) has a unique minimiser \(\Ped\in\Pi(\mud,\mud)\). Moreover, there exist strictly positive functions \(f_\varepsilon,g_\varepsilon:\Xd\to(0,\infty)\) such that
\[
\Ped(x,y)=f_\varepsilon(x)\,g_\varepsilon(y)\,e^{-\Cd(x,y)/\varepsilon}\,\Lambda^{\mathrm d}(x,y),
\qquad \Lambda^{\mathrm d}:=\mud\otimes\mud.
\]
\end{proposition}
\begin{proof}
By Lemma~\ref{lem:Hmu}, the finiteness assumption from the general discussion in Section~\ref{sec:intro} holds for the discrete model. If
\[
dR_\varepsilon^{\mathrm d}:=(Z_\varepsilon^{\mathrm d})^{-1}e^{-\Cd/\varepsilon}\,d\Lambda^{\mathrm d},
\]
then the minimisation of \(J_\varepsilon^{\mathrm d}\) is equivalent to minimising \(\varepsilon\,\KL(\cdot\|R_\varepsilon^{\mathrm d})\). By \cite[Prop.~A.1]{BGN}, existence, uniqueness, and the factorisation hold. Since \eqref{eq:mu} gives \(\mud(x)>0\) for every \(x\in\Xd\), we have \(\Lambda^{\mathrm d}(x,y)>0\) for all \(x,y\in\Xd\), so \(\Ped(x,y)>0\) is true for all \(x,y\in\Xd\).
\end{proof}

\begin{lemma}[\(2\times2\) cross-ratio identity]\label{lem:cross-ratio-proof-app}
Let \(\varepsilon>0\) and \(P=\Ped\). For all \(x_1,x_2,y_1,y_2\in\Xd\),
\[
\begin{aligned}
&P(x_1,y_1)P(x_2,y_2)\,
  e^{(\Cd(x_1,y_1)+\Cd(x_2,y_2))/\varepsilon}\\
&\quad=
P(x_1,y_2)P(x_2,y_1)\,
  e^{(\Cd(x_1,y_2)+\Cd(x_2,y_1))/\varepsilon}.
\end{aligned}
\]
\end{lemma}
\begin{proof}
We set:
\[
D:=\Cd(x_1,y_1)+\Cd(x_2,y_2)-\Cd(x_1,y_2)-\Cd(x_2,y_1).
\]
By Proposition~\ref{prop:exist-unique}, it holds that
\[
\frac{P(x_1,y_1)P(x_2,y_2)}{P(x_1,y_2)P(x_2,y_1)}
=\exp\!\Big(-\frac{D}{\varepsilon}\Big)
\frac{\Lambda^{\mathrm d}(x_1,y_1)\Lambda^{\mathrm d}(x_2,y_2)}
     {\Lambda^{\mathrm d}(x_1,y_2)\Lambda^{\mathrm d}(x_2,y_1)}.
\]
From \(\Lambda^{\mathrm d}=\mud\otimes\mud\), we deduce that the ratio is \(1\), and the claim follows immediately.
\end{proof}

\begin{lemma}[Symmetries and uniqueness]\label{lem:symmetry}
Let \(T:\Xd\to\Xd\) be a bijection with \(T_\#\mud=\mud\) and \(\Cd(Tx,Ty)=\Cd(x,y)\) for all \(x,y\). It follows that \((T\times T)_\#\Ped=\Ped\) for every \(\varepsilon>0\).
Moreover, if the two marginals coincide \((\mud,\mud)\) and \(\Cd\) is symmetric, then \(S(x,y):=(y,x)\) also satisfies \(S_\#\Ped=\Ped\).
\end{lemma}
\begin{proof}
Notice that the pushforward by \(T\times T\) preserves \(\Pi(\mud,\mud)\), the cost term, and \(\KL(\cdot\|\Lambda^{\mathrm d})\), which also implies that \(J_\varepsilon^{\mathrm d}\) is preserved. By uniqueness of the minimiser, we obtain \((T\times T)_\#\Ped=\Ped\). We could apply the same argument with the map \(S\) to show \(S_\#\Ped=\Ped\) when the marginals are \((\mud,\mud)\) and \(\Cd\) is symmetric.
\end{proof}

\begin{lemma}[A high-multiplicity pair]\label{lem:o-vs-d-app}
Let us fix \(\varepsilon>0\) and \(n\ge1\), and define
\[
d_{n,\varepsilon}:=P^{\mathrm d}_\varepsilon((n,1),(n,1)),
\qquad
o_{n,\varepsilon}:=P^{\mathrm d}_\varepsilon((n,1),(n,2)).
\]
We then obtain:
\[
o_{n,\varepsilon}=d_{n,\varepsilon}e^{-b/\varepsilon}.
\]
In particular, for \(\varepsilon_n=1/n\), it holds that
\[
o_n=d_n e^{-bn}.
\]
\end{lemma}
\begin{proof}
It suffices to apply Lemma~\ref{lem:cross-ratio-proof-app} with
\[
x_1=y_1=(n,1),\qquad x_2=y_2=(n,2).
\]
It follows that \(\Cd((n,1),(n,1))=\Cd((n,2),(n,2))=0\) and
\(\Cd((n,1),(n,2))=\Cd((n,2),(n,1))=b\), so
\[
\begin{aligned}
&P^{\mathrm d}_\varepsilon((n,1),(n,1))
 P^{\mathrm d}_\varepsilon((n,2),(n,2))\\
&\quad=
P^{\mathrm d}_\varepsilon((n,1),(n,2))
P^{\mathrm d}_\varepsilon((n,2),(n,1))\,e^{2b/\varepsilon}.
\end{aligned}
\]
By Lemma~\ref{lem:symmetry}, the diagonal terms are both \(d_{n,\varepsilon}\) and the off-diagonal terms are both \(o_{n,\varepsilon}\). We have \(d_{n,\varepsilon},o_{n,\varepsilon}>0\) by Proposition~\ref{prop:exist-unique}, so the claim follows.
\end{proof}

\section{Proof of Proposition~\ref{prop:block-estimate}}\label{app:block-bookkeeping}
\begin{proof}
\begin{enumerate}[label=\textup{(\alph*)},ref=\textup{(\alph*)},leftmargin=2.2em]
\item\label{item:block-proof-a}
By optimality of \(P_n^{\mathrm d}\), we obtain
\[
J_{\varepsilon_n}^{\mathrm d}(P_n^{\mathrm d})\le J_{\varepsilon_n}^{\mathrm d}(\Pdstar)=\varepsilon_n H(\mud).
\]
From \(\KL(P_n^{\mathrm d}\|\Lambda^{\mathrm d})\ge0\), we infer \(\sum_{x,y} \Cd(x,y)P_n^{\mathrm d}(x,y)\le \varepsilon_n H(\mud)\).
If \(x\in B_n\) and \(y\in B_n^c\), then \(y\in B_m\) for some \(m\neq n\) and \(\Cd(x,y)=L_n+L_m\ge L_n\). It follows that
\[
\sum_{x\in B_n,\,y\in B_n^c} \Cd(x,y)P_n^{\mathrm d}(x,y)\ge L_n\,\Delta_n.
\]
Moreover, given that both marginals equal \(\mud\),
\[
P_n^{\mathrm d}(B_n\times B_n^c)=\mud(B_n)-P_n^{\mathrm d}(B_n\times B_n)=P_n^{\mathrm d}(B_n^c\times B_n)=\Delta_n,
\]
and on \(B_n^c\times B_n\), we also have \(\Cd\ge L_n\).
Finally, we obtain \(\sum \Cd\,dP_n^{\mathrm d}\ge 2L_n\Delta_n\), so \(\Delta_n\le \varepsilon_n H(\mud)/(2L_n)\).

\item\label{item:block-proof-b}
We fix \(n\) and let \(\tau\) be an arbitrary permutation of \(\{1,\dots,m_n\}\). We also define a bijection \(T_\tau:\Xd\to\Xd\) by
\[
\begin{aligned}
T_\tau(n,0)&=(n,0),\\
T_\tau(n,i)&=(n,\tau(i))\qquad (1\le i\le m_n),\\
T_\tau(x)&=x\qquad (x\notin B_n).
\end{aligned}
\]
It follows that \(T_{\tau\#}\mud=\mud\) and \(\Cd(T_\tau x,T_\tau y)=\Cd(x,y)\) for all \(x,y\in\Xd\). By Lemma~\ref{lem:symmetry}, \((T_\tau\times T_\tau)_\#P_n^{\mathrm d}=P_n^{\mathrm d}\). We infer that \(\sum_{y\notin B_n}P_n^{\mathrm d}((n,i),y)\) is independent of \(i\in\{1,\dots,m_n\}\). By definition this common value is \(e_n\), so summing over \(i=1,\dots,m_n\) gives
\[
m_n e_n\le \Delta_n.
\]
We now apply \ref{item:block-proof-a} to obtain
\begin{equation}\label{eq:e-bound}
e_n\le \frac{\Delta_n}{m_n}\le \frac{\varepsilon_n H(\mud)}{2L_n m_n}.
\end{equation}

Notice that by swap symmetry (Lemma~\ref{lem:symmetry} applied to \(S(x,y)=(y,x)\)), it follows that
\[
P_n^{\mathrm d}((n,0),(n,1))=P_n^{\mathrm d}((n,1),(n,0))=r_n.
\]
Let us apply the within-block permutation symmetry with permutations \(\tau\) that satisfy \(\tau(1)=i\). We then obtain
\[
P_n^{\mathrm d}((n,0),(n,i))=P_n^{\mathrm d}((n,0),(n,1))=r_n,\qquad 1\le i\le m_n.
\]
Finally, we have
\[
\sum_{i=1}^{m_n}P_n^{\mathrm d}((n,0),(n,i))=m_n r_n\le \mud(\{(n,0)\})=\frac{w_n}{m_n+1},
\]
so
\begin{equation}\label{eq:r-bound}
r_n\le \frac{w_n}{m_n(m_n+1)}.
\end{equation}

Observe that the row sum at \((n,1)\) is
\begin{equation}\label{eq:rowsum}
\mud(\{(n,1)\}) = d_n + (m_n-1)o_n + r_n + e_n.
\end{equation}
From \(o_n=d_n e^{-bn}\), we deduce
\[
d_n\Big(1+(m_n-1)e^{-bn}\Big)=\mud(\{(n,1)\})-r_n-e_n.
\]

We shall begin by bounding the numerator. By \eqref{eq:r-bound}, for all sufficiently large \(n\), we get \(m_n\ge4\), so
\[
r_n\le \frac{w_n}{m_n(m_n+1)}\le \frac14\,\frac{w_n}{m_n+1}=\frac14\,\mud(\{(n,1)\}).
\]
We also apply \eqref{eq:e-bound} and \(\mud(\{(n,1)\})=w_n/(m_n+1)\) to obtain:
\[
\frac{e_n}{\mud(\{(n,1)\})}
\le \frac{\varepsilon_n H(\mud)}{2L_n m_n}\cdot\frac{m_n+1}{w_n}
\le \frac{\varepsilon_n H(\mud)}{L_n w_n}
\xrightarrow[n\to\infty]{}0.
\]
In the last limit, we used \eqref{eq:concrete-seqs} (see Lemma~\ref{lem:Hmu-proof-app}). This implies that for all sufficiently large \(n\),
\[
e_n\le \frac14\,\mud(\{(n,1)\}).
\]
Therefore,
\[
\mud(\{(n,1)\})-r_n-e_n\ge \frac12\,\mud(\{(n,1)\}).
\]

For the denominator, \(m_n e^{-\kappa n}\to1\) and \(b>\kappa\) imply \((m_n-1)e^{-bn}\to0\). In particular, for all sufficiently large \(n\), it holds that
\[
1+(m_n-1)e^{-bn}\le 2.
\]
It now suffices to combine the numerator and denominator bounds as to obtain
\[
d_n\ge \frac{\frac12\,\mud(\{(n,1)\})}{2}=\frac14\,\mud(\{(n,1)\}).
\]

\item
By symmetry, we notice that each ordered pair \(((n,i),(n,j))\) with \(1\le i\neq j\le m_n\) has mass \(o_n\). Because \(F_n\) contains \(m_n(m_n-1)\) such pairs, \(P_n^{\mathrm d}(F_n)=m_n(m_n-1)o_n\) follows. Let us apply \(o_n=d_n e^{-bn}\) and \ref{item:block-proof-b}. We then have
\[
P_n^{\mathrm d}(F_n)=m_n(m_n-1)d_n e^{-bn}
\ge m_n(m_n-1)\,\frac{w_n}{4(m_n+1)}\,e^{-bn}.
\]
For large \(n\), \(m_n\ge 3\), so \((m_n-1)/(m_n+1)\ge 1/2\). Therefore,
\[
P_n^{\mathrm d}(F_n)\ge \frac{w_n}{8}\,m_n e^{-bn}.
\]
Since \(m_n e^{-\kappa n}\to1\), for all large \(n\), we have \(m_n\ge \frac12 e^{\kappa n}\), so
\[
m_n e^{-bn}\ge \frac12\,e^{-(b-\kappa)n}
\]
for large \(n\). Finally, for large \(n\),
\[
P_n^{\mathrm d}(F_n)\ge \frac{w_n}{16}\,e^{-(b-\kappa)n}.
\]
\end{enumerate}
\end{proof}

\section{Verifications for the discrete model}\label{app:routine-checks}
\begin{lemma}[Compact sets in discrete spaces]\label{lem:compact-finite}
If \(\Omega\) is discrete, then every compact subset \(E\subset\Omega\) is finite.
\end{lemma}
\begin{proof}
Suppose that \(E\) were infinite. Let us also choose pairwise distinct \(x_n\in E\). In a discrete space, there does not exist an infinite sequence of distinct points with a convergent subsequence, which contradicts compactness.
\end{proof}

\begin{lemma}[Verification of basic properties of \eqref{eq:concrete-seqs}]\label{lem:Hmu-proof-app}
For the sequences \eqref{eq:concrete-seqs} and \(\varepsilon_n:=1/n\),
\[
\begin{gathered}
\sum_{n\ge1} w_n=1,\qquad
\sum_{n\ge1} w_n\log\frac{m_n+1}{w_n}<\infty,\\
L_n\to\infty,\qquad
m_n e^{-\kappa n}\to1,
\end{gathered}
\]
and
\[
\frac{\varepsilon_n}{L_n w_n}\to0,
\qquad
\frac1n\log w_n\to0.
\]
\end{lemma}
\begin{proof}
Observe that the normalisation \(\sum_n w_n=1\) follows immediately from \(w_n=\zeta(3)^{-1}n^{-3}\), and \(L_n=n^4\to\infty\). Furthermore,
\[
\frac{\varepsilon_n}{L_n w_n}
=\frac{1/n}{n^4\cdot (\zeta(3)^{-1}n^{-3})}
=\frac{\zeta(3)}{n^2}\xrightarrow[n\to\infty]{}0.
\]
Since \(m_n=\lfloor e^{\kappa n}\rfloor\) for all large \(n\), we have \(m_n e^{-\kappa n}\to1\). Also,
\[
\frac1n\log w_n
=-\frac{3\log n+\log\zeta(3)}{n}\xrightarrow[n\to\infty]{}0.
\]
For the entropy term, it is true that
\[
\sum_{n\ge1} w_n\log\frac{m_n+1}{w_n}
=\sum_{n\ge1} w_n\log(m_n+1)+\sum_{n\ge1} w_n\log\frac1{w_n}.
\]
From \(m_n+1\le 3e^{\kappa n}\), we get
\[
\sum_{n\ge1} w_n\log(m_n+1)
\le \frac{1}{\zeta(3)}\sum_{n\ge1}\frac{\kappa n+\log 3}{n^3}<\infty.
\]
Finally, \(w_n=\zeta(3)^{-1}n^{-3}\) implies
\[
\sum_{n\ge1} w_n\log\frac1{w_n}
=\frac{1}{\zeta(3)}\sum_{n\ge1}\frac{\log\zeta(3)+3\log n}{n^3}<\infty.
\]
We conclude that \(\sum_n w_n\log((m_n+1)/w_n)<\infty\).
\end{proof}

\bibliographystyle{amsplain}
\bibliography{open_relaxation2}
\end{document}